\documentclass{amsart}
\usepackage{calc,amsfonts,amsmath,amssymb,amsthm,graphicx}
\usepackage[sort&compress,numbers]{natbib}
\usepackage[all]{xy}

\newcommand{\thmlabel}[1]{\label{thm:#1}} 
\newcommand{\thmref}[1]{Theorem~\ref{thm:#1}}

\newcommand{\lemlabel}[1]{\label{lem:#1}}
\newcommand{\lemref}[1]{Lemma~\ref{lem:#1}}
\newcommand{\twolemref}[2]{Lemmas~\ref{lem:#1} and \ref{lem:#2}}

\newcommand{\eqnlabel}[1]{\label{eqn:#1}}
\newcommand{\eqnref}[1]{\eqref{eqn:#1}}

\newcommand{\figlabel}[1]{\label{fig:#1}}
\newcommand{\figref}[1]{Figure~\ref{fig:#1}}

\newcommand{\seclabel}[1]{\label{sec:#1}}
\newcommand{\secref}[1]{Section~\ref{sec:#1}}

\newcommand{\corlabel}[1]{\label{cor:#1}}
\newcommand{\coref}[1]{Corollary~\ref{cor:#1}}

\newcommand{\proplabel}[1]{\label{prop:#1}}

\newcommand{\conjlabel}[1]{\label{con:#1}}

\theoremstyle{plain}
\newtheorem{theorem}{Theorem}[section]
\newtheorem{lemma}[theorem]{Lemma}
\newtheorem{corollary}[theorem]{Corollary}
\newtheorem{proposition}[theorem]{Proposition}
\newtheorem{open}[theorem]{Open Problem}

\theoremstyle{definition}
\newtheorem{conjecture}[theorem]{Conjecture}

\newcommand{\Figure}[4][htb]{
\begin{figure}[#1]
	\begin{center}#3\end{center}
	\caption{\figlabel{#2}#4}
\end{figure}}

\newcommand{\blah}[1]{\ensuremath{\protect\langle#1\rangle}}
\newcommand{\ceil}[1]{\ensuremath{\protect\lceil#1\rceil}}
\newcommand{\Oh}[1]{\ensuremath{\protect\mathcal{O}(#1)}}
\newcommand{\bracket}[1]{\ensuremath{\protect\left(#1\right)}}

\newcommand{\half}{\ensuremath{\protect\tfrac{1}{2}}}

\newcommand{\CEIL}[1]{\ensuremath{\protect\left\lceil#1\right\rceil}}

\newcommand{\jn}[2][]{\ensuremath{\textup{\textsf{jn}}_{#1}(#2)}}
\newcommand{\qn}[2][]{\ensuremath{\textup{\textsf{qn}}_{#1}(#2)}}
\newcommand{\sn}[2][]{\ensuremath{\textup{\textsf{sn}}_{#1}(#2)}}

\newcommand{\st}{\ensuremath{\chi_{\textup{\textsf{st}}}}}
\newcommand{\acy}{\ensuremath{\chi_{\textup{\textsf{a}}}}}
\newcommand{\CR}[1]{\ensuremath{\protect\textsf{\textup{cr}}(#1)}}

\newcommand{\N}{\ensuremath{\mathbb{N}}}
\newcommand{\R}{\ensuremath{\mathbb{R}}}

\DeclareMathOperator{\col}{col}
\DeclareMathOperator{\dist}{dist}
\newcommand{\shm}{\,\triangledown\,}
\newcommand{\shtm}{\,\widetilde{\triangledown}\,}
\newcommand{\bbbn}{\N}
\newcommand{\bbbr}{\R}
\newcommand{\rdens}[1]{\ensuremath{\nabla_{#1}}}
\newcommand{\trdens}[1]{\ensuremath{\widetilde{\nabla}_{#1}}}
\def\rsub{\rotatebox[origin=c]{90}{$\subseteq$}}
\def\req{\rotatebox[origin=c]{90}{$=$}}

\newcommand{\card}[1]{|#1|}
\newcommand{\nlongrightarrow}{\relbar\!\joinrel\not\relbar\joinrel\!\!\rightarrow}

\begin{document}

\title[Graph Classes with Bounded Expansion]{Characterisations and Examples of\\ Graph Classes with Bounded Expansion}

\author{Jaroslav Ne{\v{s}}et{\v{r}}il}
\address{Department of Applied Mathematics,\\ and
Institute for Theoretical Computer Science\\ 
Charles University\\ 
Prague, Czech Republic}
\email{nesetril@kam.mff.cuni.cz}

\author{Patrice Ossona de Mendez}
\address{Centre d'Analyse et de Math\'ematique Sociales \\
Centre National de la Recherche Scientifique, and \\
\'Ecole des Hautes \'Etudes en Sciences Sociales\\ 
Paris, France}
\email{pom@ehess.fr}

\author{David R.~Wood}
\address{Department of Mathematics and Statistics\\ 
The University of Melbourne\\ 
Melbourne, Australia}
\email{woodd@unimelb.edu.au}

\begin{abstract}
Classes with bounded expansion, which generalise classes that exclude a topological minor, have recently been introduced by Ne\v{s}et\v{r}il and Ossona~de~Mendez. These classes are defined by the fact that the maximum average degree of a shallow minor of a graph in the class is bounded by a function of the depth of the shallow minor. Several linear-time algorithms are known for bounded expansion classes (such as subgraph isomorphism testing), and they allow restricted homomorphism dualities, amongst other desirable properties.

In this paper we establish two new characterisations of bounded expansion classes, one in terms of so-called topological parameters, the other in terms of controlling dense parts. The latter characterisation is then used to show that the notion of bounded expansion is compatible with Erd\"os-R\'enyi model of random graphs with constant average degree. In particular, we prove that for every fixed $d>0$, there exists a class with bounded expansion, such that a random graph of order $n$ and edge probability $d/n$ asymptotically almost surely belongs to the class. 

We then present several new examples of classes with bounded expansion that do not exclude some topological minor, and appear naturally in the context of graph drawing or graph colouring. In particular, we prove that the following classes have bounded expansion: graphs that can be drawn in the plane with a bounded number of crossings per edge, graphs with bounded stack number, graphs with bounded queue number, and graphs with bounded non-repetitive chromatic number. We also prove that graphs with `linear' crossing number are contained in a topologically-closed class, while graphs with bounded crossing number are contained in a minor-closed class.
\end{abstract}

\keywords{graph,  queue layout,  queue-number,  stack layout,  stack-number,  book embedding,  book thickness,  page-number,  expansion,  bounded expansion,  crossing number,  non-repetitive chromatic number,  Thue number,  random graph,   jump number}

\subjclass{05C62 (graph representations), 05C15 (graph coloring), 05C83 (graph minors)}

\maketitle

\newpage
\tableofcontents
\newpage

\section{Introduction}




What is a `sparse' graph? It is not enough to simply consider edge density as the measure of sparseness. For example, if we start with a dense graph (even a complete graph) and subdivide each edge by inserting a new vertex, then in the obtained graph the number of edges is less than twice the number of vertices. Yet in several aspects, the new graph inherits the structure of the original.

A natural restriction is to consider \emph{proper minor-closed} graph classes. These are the classes of graphs that are closed under vertex deletions, edge deletions, and edge contractions (and some graph is not in the class). Planar graphs are a classical example. Interest in minor-closed classes is widespread. Most notably, \citet{RS-GraphMinors} proved that every minor-closed class is characterised by a \emph{finite} set of excluded minors. (For example, a graph is planar if and only if it has no $K_5$-minor and no $K_{3,3}$-minor.)\ Moreover, membership in a particular minor-closed class can be tested in polynomial time. 
There are some limitations however in using minor-closed classes as models for sparse graphs. 
For example, cloning each vertex (and its incident edges) does not preserve such properties.
In particular, the graph obtained by cloning each vertex in the $n\times n$ planar grid graph has unbounded clique minors \citep{Wood-ProductMinor}.



A more general framework concerns \emph{proper topologically-closed} classes of graphs. These classes are characterised as follows: whenever a subdivision of a graph $G$ belongs to the class then $G$ also belongs to the class (and some graph is not in the class). Such a class is characterised by a possibly infinite set of forbidden configurations. 

A further generalisation consists in classes of graphs having \emph{bounded expansion}, as introduced by 
\citet{ICGT05,Taxi_stoc06,NesOdM-GradI}. Roughly speaking, these classes are defined by the fact that the maximum average degree of a shallow minor of a graph in the class is bounded by a function of the depth of the shallow minor. Thus  bounded expansion classes are broader than minor-closed classes, which are those classes for which every minor of every graph in the class has bounded average degree. 

Bounded expansion classes have a number of desirable properties. (For an extensive study we refer the reader to \citep{NesOdM-GradI,NesOdM-GradII,NesOdM-GradIII,Dvo,Dvorak-EUJC08}.)\ For example, they admit so-called {\em low tree-depth decompositions} \citep{NesOdM-TreeDepth-EJC06}, which extend the low tree-width decompositions introduced by \citet{DDOSRSV-JCTB04} for minor-closed classes. These decompositions, which may be computed in linear time, are at the core of several linear-time graph algorithms, such as testing for an induced subgraph isomorphic to a fixed pattern \citep{Taxi_stoc06,NesOdM-GradII}. In fact, isomorphs of a fixed pattern graph can be counted in a graph from a bounded expansion class in linear time \citep{Taxi_hom}. Also, low tree-depth decompositions imply the existence of restricted homomorphism dualities for classes with bounded expansion \citep{NesOdM-GradIII}. That is, for every class $\mathcal C$ with bounded expansion and every connected graph $F$ (which is not necessarily in $\mathcal C$) there exists a graph $D_{\mathcal C}(F)$ such that
\begin{equation*}
\forall G\in\mathcal C:\qquad (F\nlongrightarrow G)\iff (G\longrightarrow D_{\mathcal C}(F))\enspace,
\end{equation*}
where $G\rightarrow H$ means that there is a homomorphism from $G$ to $H$, and $G\nrightarrow H$ means that there is no such homomorphism.  Finally, note that the structural properties of bounded expansion classes make them particularly interesting as a model in the study of `real-world' sparse networks \citep{Aeolus06}.

Bounded expansion classes are the focus of this paper. Our contributions to this topic are classified as follows (see \figref{bec}):
\begin{itemize}
\item We establish two new characterisations of bounded expansion classes, one in terms of so-called topological parameters, the other in terms of controlling dense parts; see  \secref{Characterisations}.
\item This latter characterisation is then used to show that the notion of bounded expansion is compatible with Erd\"os-R\'enyi model of random graphs with constant average degree (that is, for random graphs of order $n$ with edge probability $d/n$); see \secref{Random}. 
\item We present several new examples of classes with bounded expansion that appear naturally in the context of graph drawing or graph colouring. In particular, we prove that each of the following classes
have bounded expansion, even though they are not contained in a proper topologically-closed class:
\begin{itemize}
\item graphs that can be drawn with a bounded number of crossings per edge (\secref{CrossingNumber}), 
\item graphs with bounded queue-number (\secref{Queue}),
\item graphs with bounded stack-number (\secref{Stack}),
\item graphs with bounded non-repetitive chromatic number (\secref{NonRep}).
\end{itemize}
We also prove that graphs with `linear' crossing number are contained in a topologically-closed class, and graphs with bounded crossing number are contained in a minor-closed class (\secref{CrossingNumber}).
\end{itemize}

\newlength{\figlength}

\begin{figure}[htb]
\begin{center}
\def\BEsize{\small}
\settowidth{\figlength}{\BEsize bounded crossings}
\xy\xymatrix @R=1pt@C=1pt{
&\includegraphics[width=0.1\textwidth]{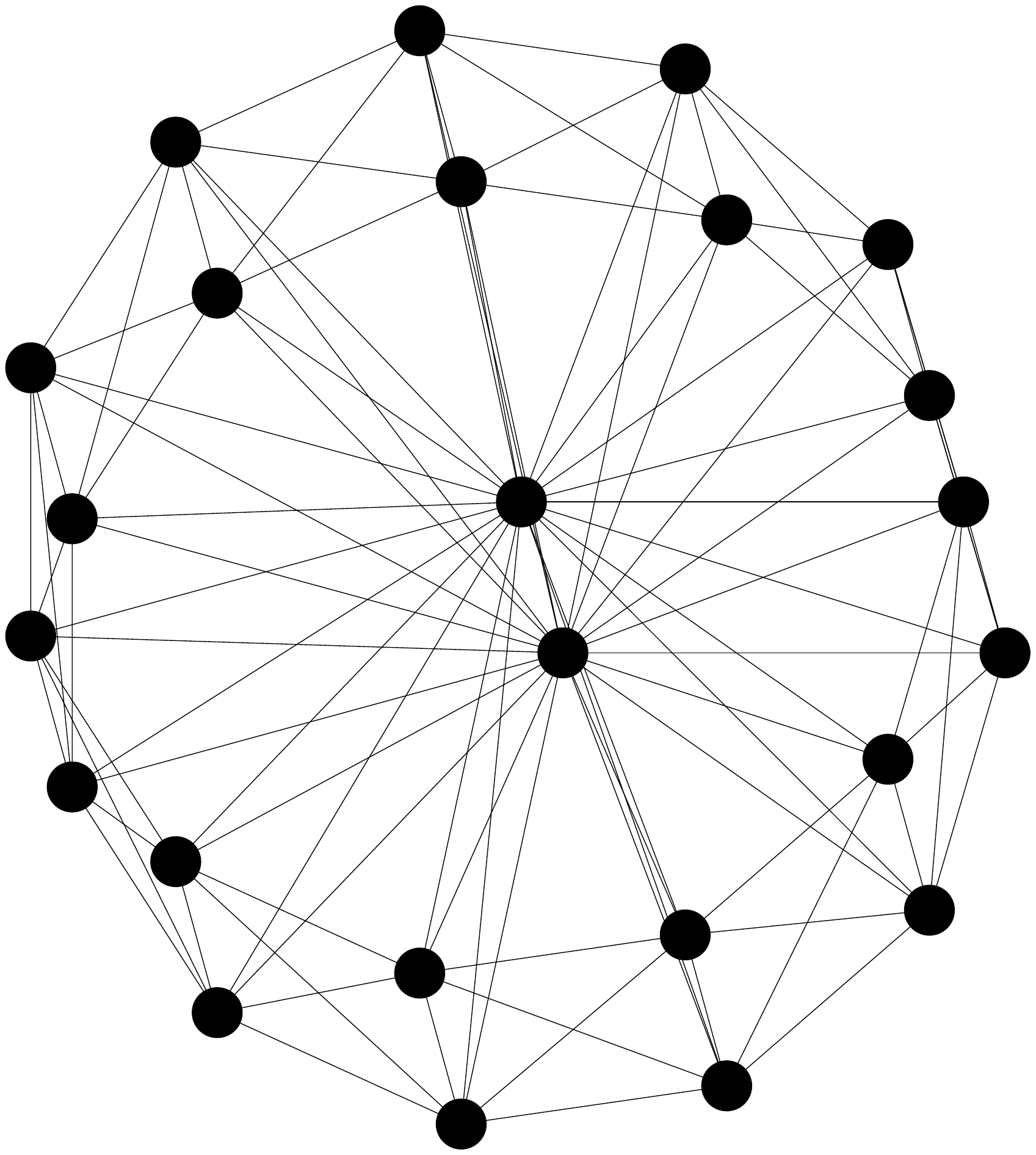}&&&\\
&\text{\BEsize bounded expansion}&&&\\
\raisebox{0pt}[10pt]{}&&&\\
\includegraphics[width=0.1\textwidth]{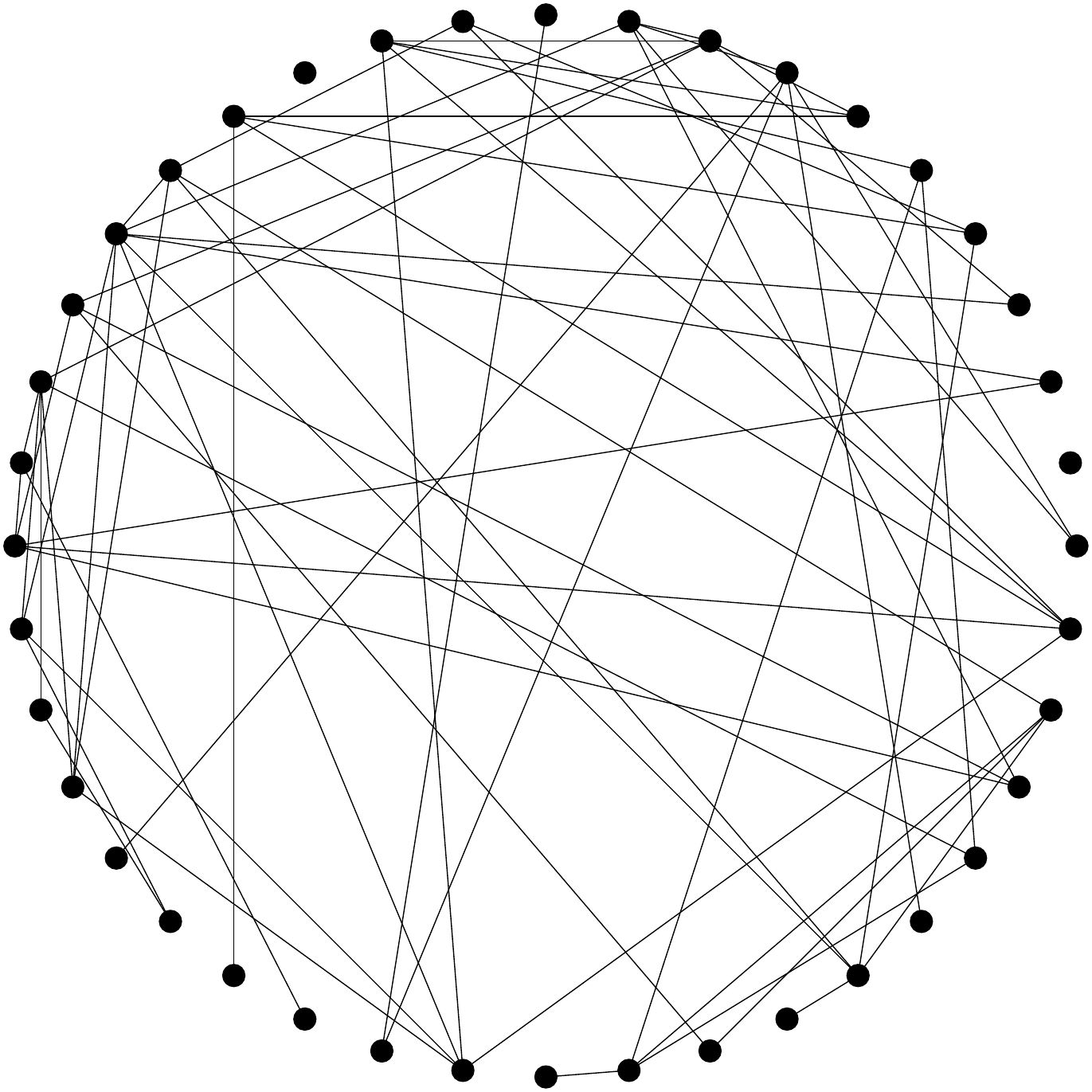}\ar[uur]&
\includegraphics[width=0.1\textwidth]{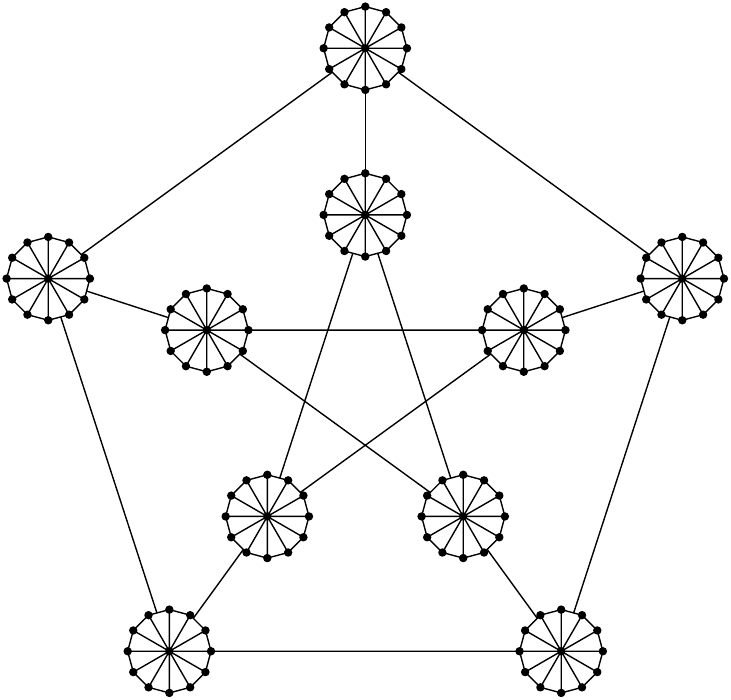}\ar[uu]&
\includegraphics[width=0.1\textwidth]{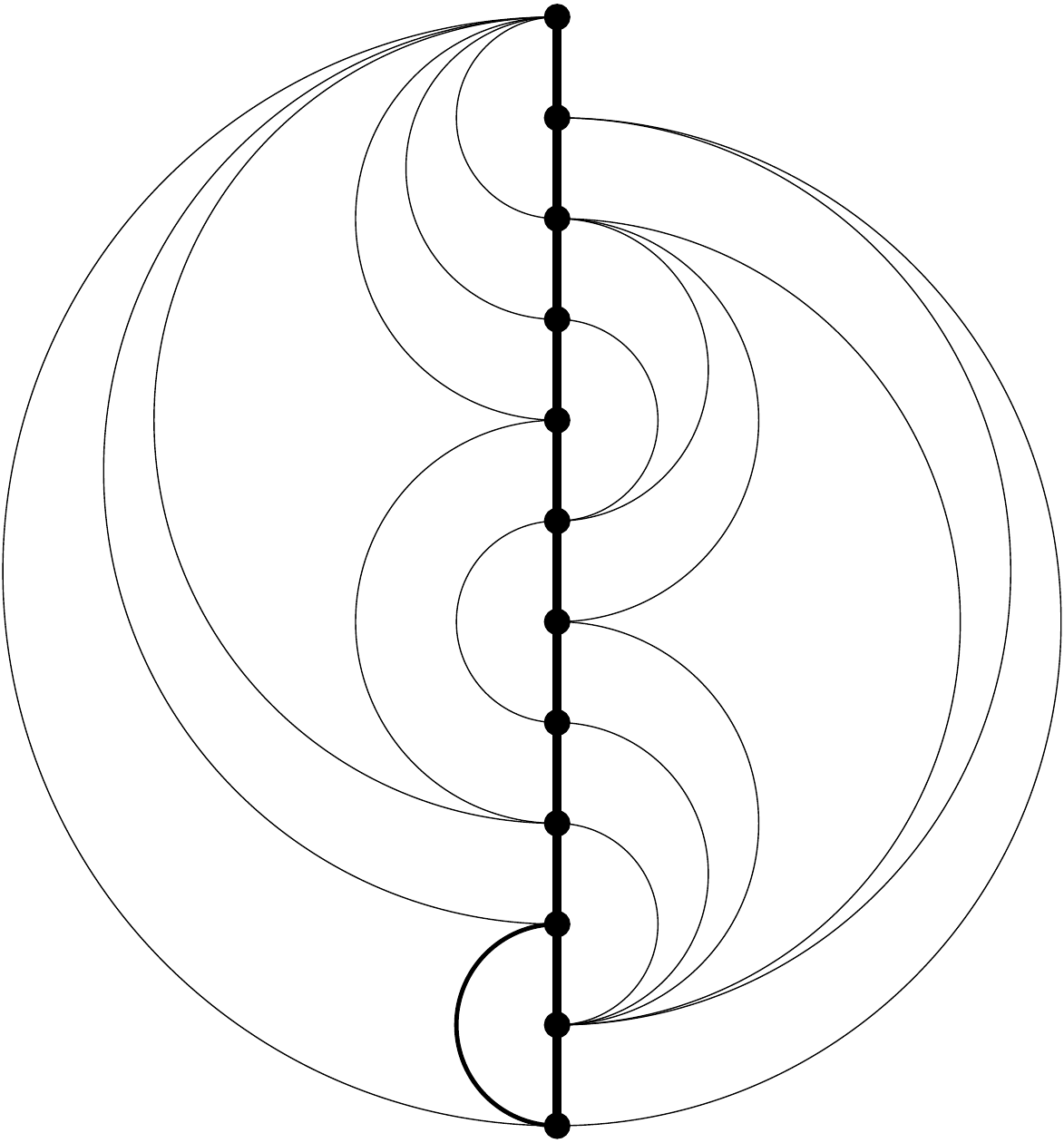}\ar[uul]& 
\includegraphics[width=0.15\textwidth]{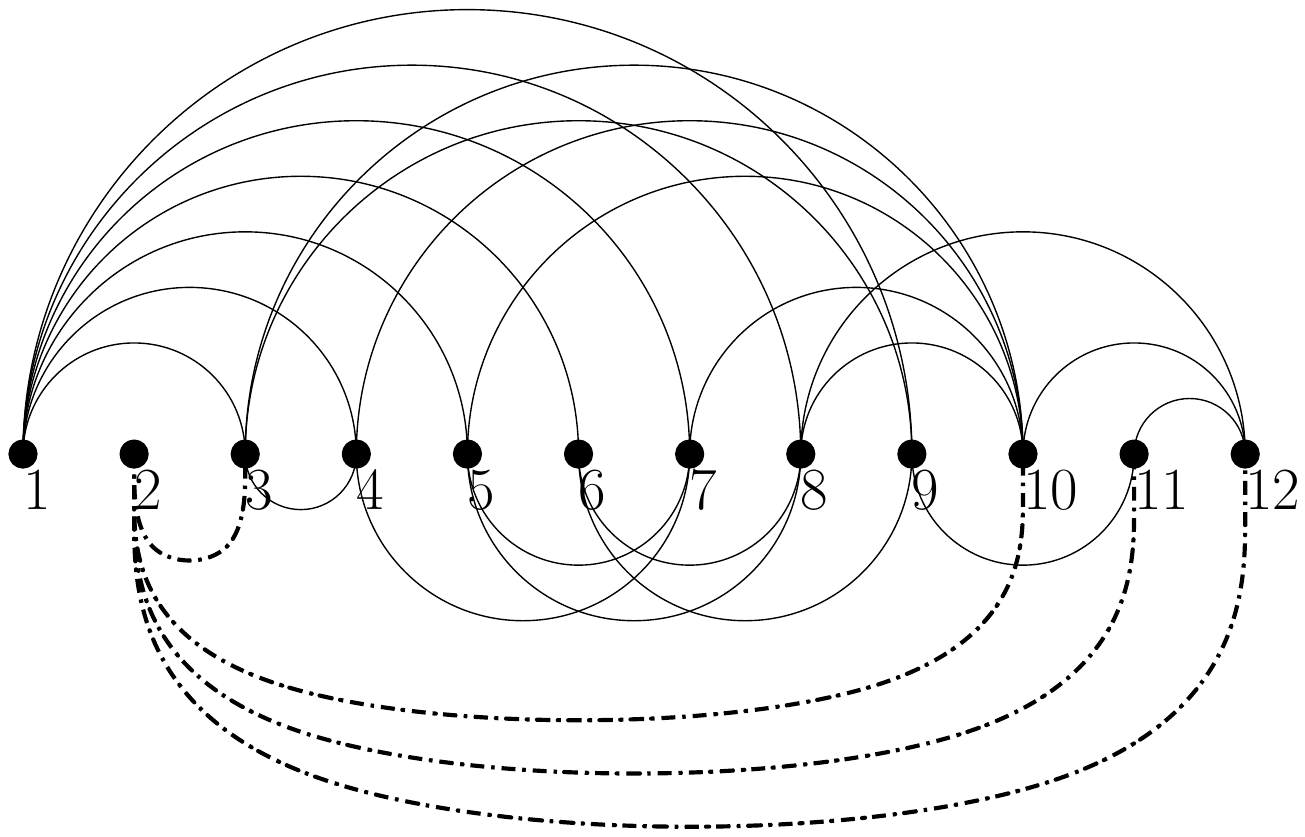}\POS[];[uull]**\crv{[uul]}?>*\dir{>}
\\
\text{\BEsize random $G(n,d/n)$}&\text{\BEsize topologically closed}&\txt<\figlength>{\BEsize bounded stack number}&
\txt<\figlength>{\BEsize bounded queue number}\\
\raisebox{0pt}[10pt]{}&&&\\
\includegraphics[width=0.1\textwidth]{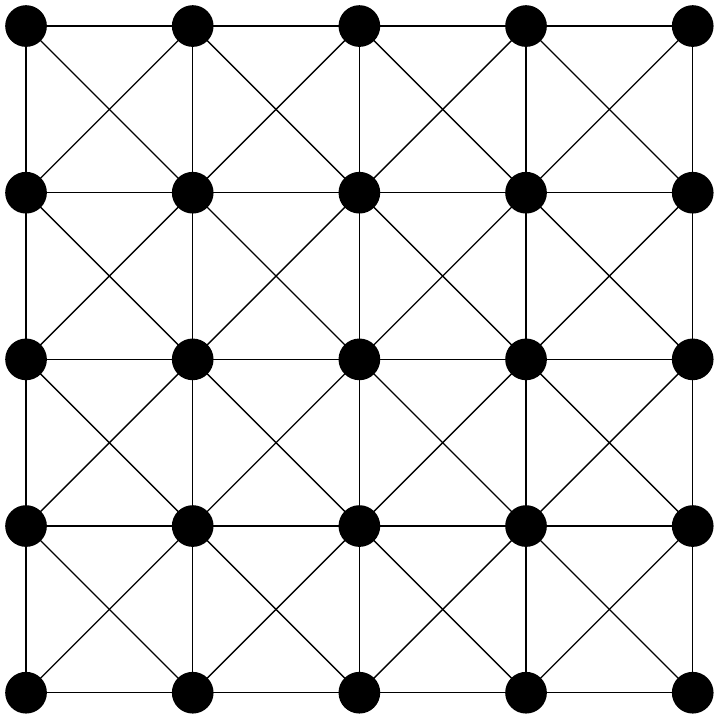}
\POS[];[uuuuur]**\crv{[uur]&[uu]}?>*\dir{>}&
\includegraphics[width=0.15\textwidth]{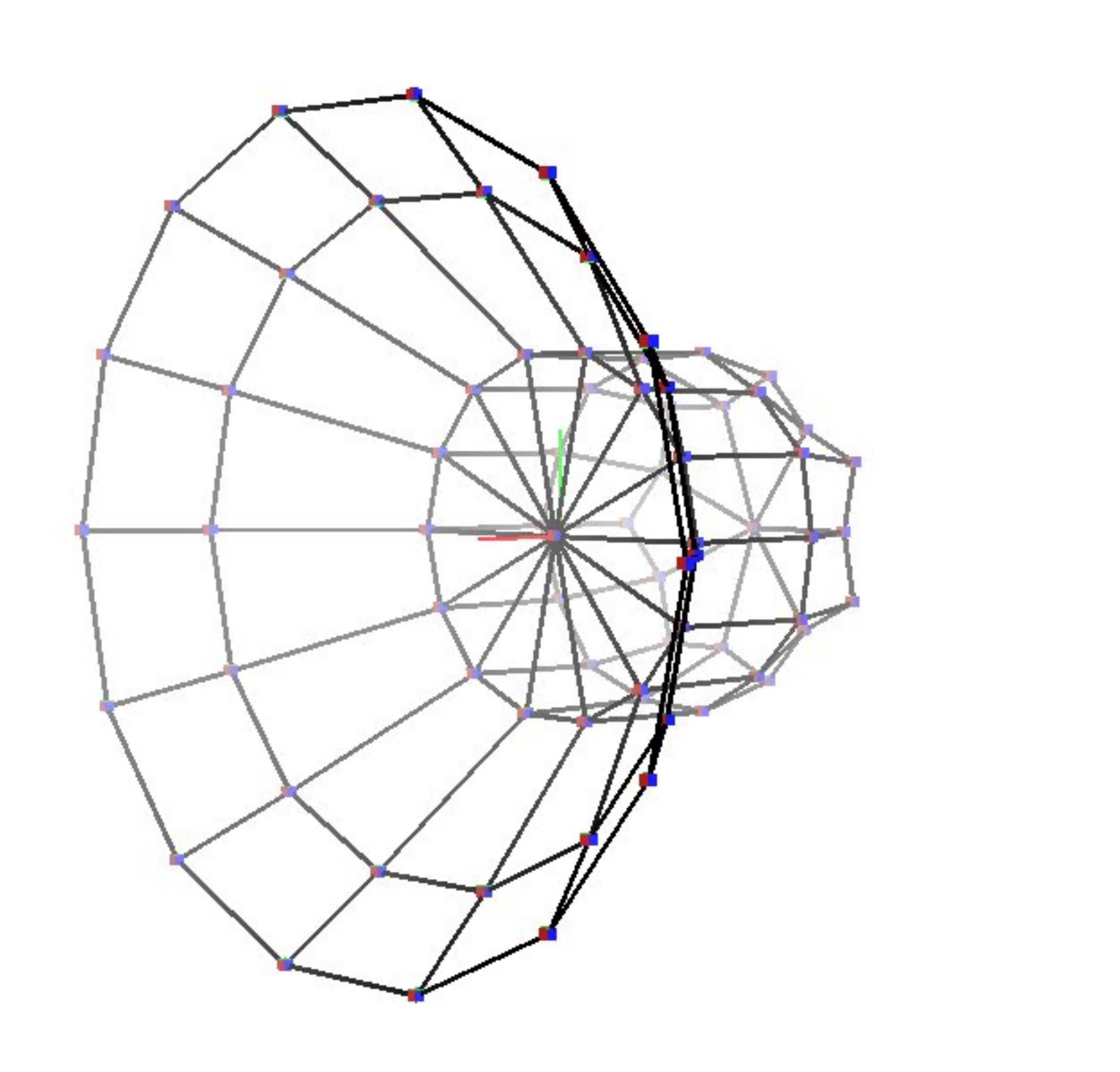}\ar[uu]&
\includegraphics[width=0.1\textwidth]{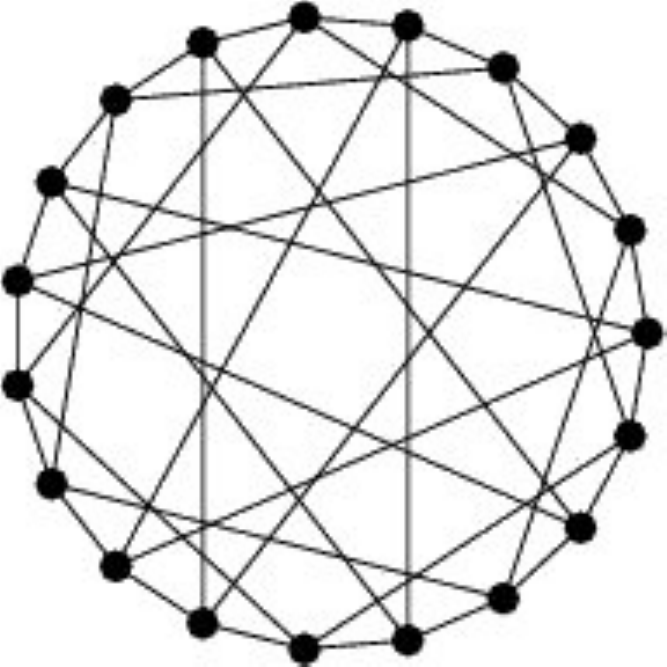} 
\ar[uul]
&\includegraphics[width=0.1\textwidth]{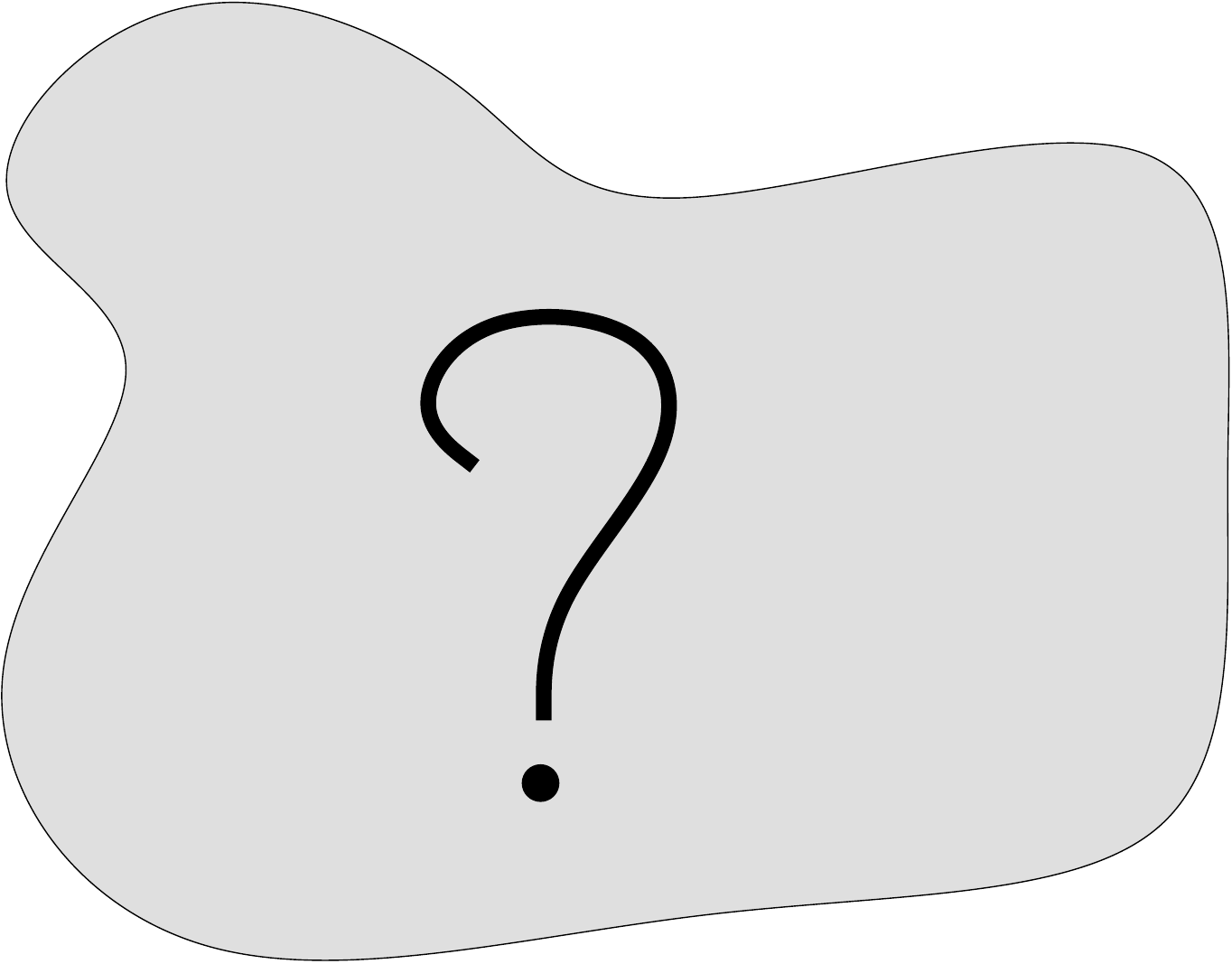}
\POS[];[uuuuull]**\crv{[uul]&[uuuu]}?>*\dir{>}
\\
\addtolength{\figlength}{10pt}\txt<\figlength>{\BEsize bounded number of crossings per edge}\addtolength{\figlength}{-10pt}
&\text{\BEsize minor closed}&
\text{\BEsize bounded degree}
&\txt<\figlength>{\BEsize non-repetitively $k$-colourable}\\
\raisebox{0pt}[10pt]{}&&&\\
\includegraphics[width=0.1\textwidth]{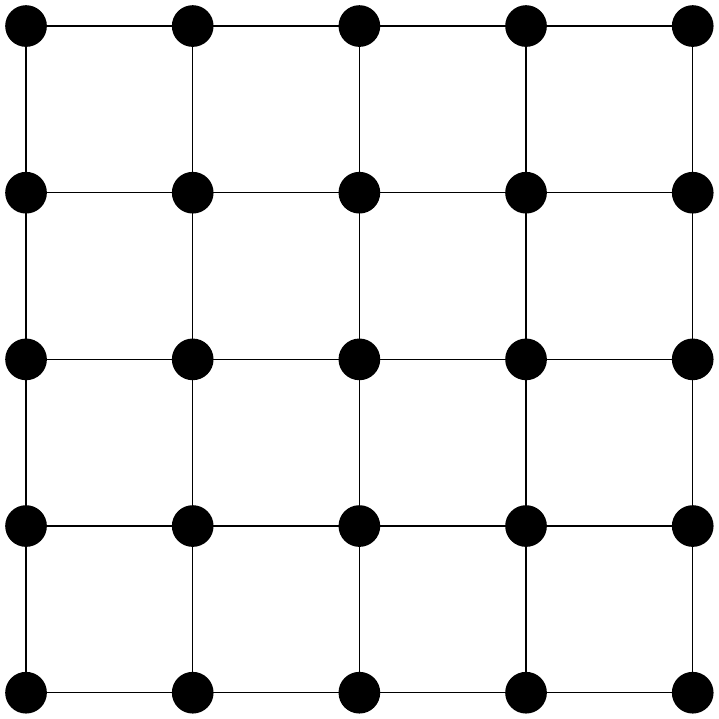}\ar[r]& 
\includegraphics[width=0.1\textwidth]{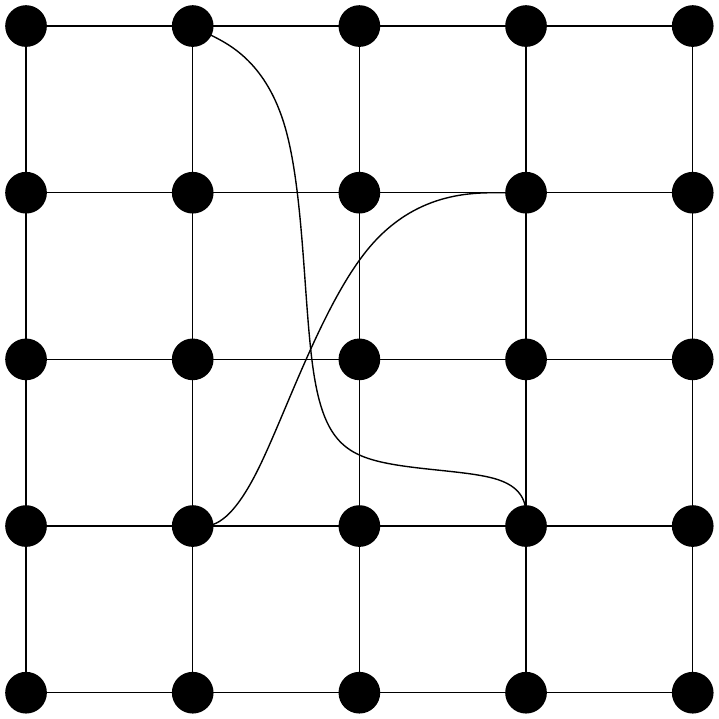}\ar[uu]\ar[uul]\ar[r]& 
\includegraphics[width=0.1\textwidth]{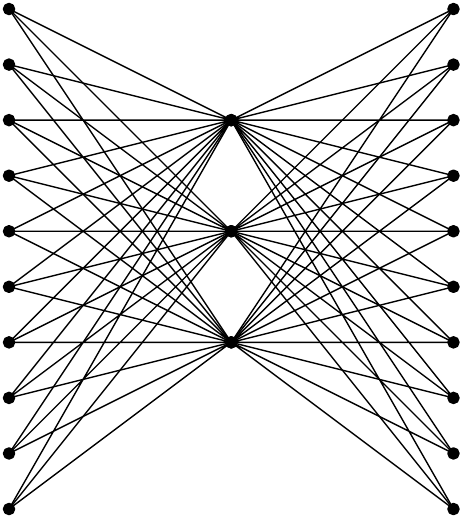}
\POS[];[uuuuul]**\crv{[ul]&[uu]}?>*\dir{>}
\\
\text{\BEsize planar}&
\txt<\figlength>{\BEsize bounded crossing number}&
\txt<\figlength>{\BEsize linear crossing number}
}\endxy
\end{center}
\caption{Classes with Bounded Expansion. The results about classes with bounded crossings, bounded queue-number, bounded stack-number, and bounded non-repetitive chromatic number are proved in this paper.}
\figlabel{bec}
\end{figure}

Before continuing we recall some well-known definitions and results about graph colourings. A \emph{colouring} of a graph $G$ is a function $f$ from $V(G)$ to some set of colours, such that $f(v)\neq f(w)$ for every edge $vw\in E(G)$. A subgraph $H$ of a coloured graph $G$ is \emph{bichromatic} if at most two colours appear in $H$. A colouring is \emph{acyclic} if there is no bichromatic cycle; that is, every bichromatic subgraph is a forest. The \emph{acyclic chromatic number} of $G$, denoted by $\acy(G)$, is the minimum number of colours in an acyclic colouring of $G$. A colouring is a \emph{star colouring} if every bichromatic subgraph is a star forest; that is, there is no bichromatic $P_4$. The \emph{star chromatic number} of $G$, denoted by $\st(G)$, is the minimum number of colours in a star colouring of $G$. Observe that a star colouring is acyclic, and $\acy(G)\leq\st(G)$ for all $G$. Conversely, the star chromatic number is bounded by a function of the acyclic chromatic number (folklore, see \cite{Fertin2004,Albertson-EJC04}). That graphs with bounded expansion have bounded star chromatic number is proved in \citep{Taxi_stoc06,NesOdM-GradI}. 

\section{Shallow Minors and Bounded Expansion Classes}

In the following, we work with unlabelled finite simple graphs.
We use standard graph theory terminology. 
In particular, for a graph $G$, we denote by $V(G)$ its vertex set, by $E(G)$ its edge set, by
$|G|$ its {\em order} (that is, $|V(G)|$) and by $\|G\|$ its {\em size} (that is, $|E(G)|$).
The {\em distance} between two vertices $x$ and $y$ of $G$, denoted by $\dist_G(x,y)$, is the minimum length (number of edges) of a path linking $x$ and $y$ (or $\infty$ if $x$ and $y$ do not belong to the same connected component of $G$). 
The {\em radius} of a connected graph $G$ is the minimum over all vertices $r$ of $G$ of the maximum distance between $r$ and another vertex of $G$. 
For a subset of vertices $A$ of $G$, the {\em subgraph of $G$ induced by} $A$ will be denoted by $G[A]$. 

A class $\mathcal C$ of graphs is {\em hereditary} if every induced subgraph of a graph in $\mathcal C$ is also in $\mathcal C$, and $\mathcal C$ is {\em monotone} if every subgraph of a graph in $\mathcal C$ is also in $\mathcal C$.

For $d\in\bbbn$, a graph $H$ is said to be a {\em shallow minor} of a graph $G$ at {\em depth} $d$ if there exists a subgraph $X$ of $G$ whose connected components have radius at most $d$, such that $H$ is a simple graph obtained from $G$ by contracting each component of $X$ into a single vertex and then taking a subgraph (see \figref{shm}).
\citet{PRS-SODA94}, who introduced shallow minors as \emph{low-depth minors}, attributed this notion to Charles Leiserson and Sivan Toledo.

\begin{figure}[htb]
\begin{center}
\includegraphics[width=0.75\textwidth]{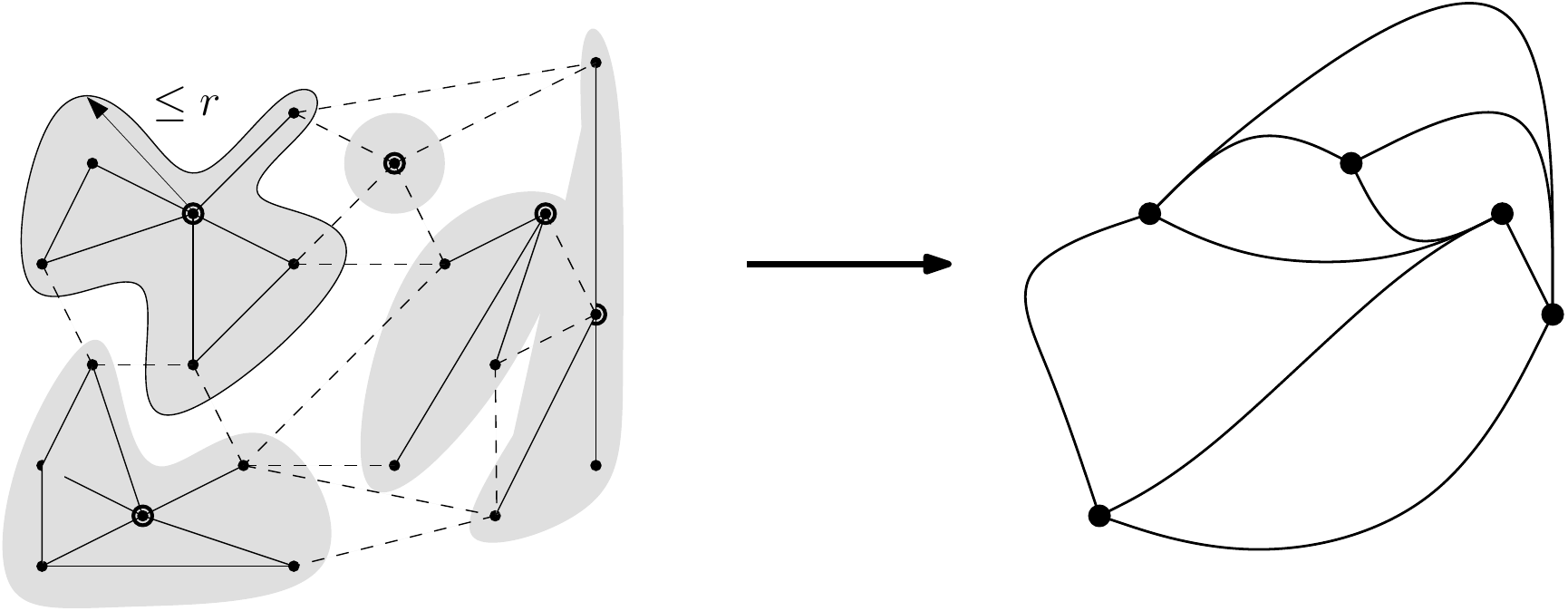}
\caption{A shallow minor of depth $d$ of a graph $G$ is a simple subgraph of a minor of $G$ obtained by contracting vertex disjoint subgraphs with radius at most $d$.}
\figlabel{shm}
\end{center}
\end{figure}

For a graph $G$ and $d\in\bbbn$, let $G\shm d$ denote the set of all shallow minors of $G$ at depth $d$. In particular, $G\shm 0$ is the set of all subgraphs of $G$. Hence we have the following non-decreasing sequence of classes: 
\begin{equation*}
G\in G\shm 0\subseteq G\shm 1\subseteq\dots\subseteq G\shm d\subseteq\dots G\shm\infty\enspace.
\end{equation*}
We extend this definition in the obvious way to graph classes $\mathcal C$ by defining
\begin{equation*}
\mathcal C\shm d=\bigcup_{G\in\mathcal C} G\shm d\enspace.
\end{equation*}

The information gained by considering shallow minors instead of minors enables robust classification of graphs classes. An infinite graph class $\mathcal C$ is said to be {\em somewhere dense} if there exists an integer $d$ such that every (finite simple) graph belongs to $\mathcal C\shm d$, otherwise $\mathcal C$ is {\em nowhere dense} \citep{NesOdM-NowhereDense,Nev2008}. That is, a graph class is somewhere dense if every graph is a bounded depth shallow minor of a graph in the class. Nowhere dense classes are closely related  to {\em quasi wide} classes \citep{ND_logic}, which were introduced in the context of First Order Logic by \citet{Dawar2007}, and to asymptotic counting of homomorphisms from fixed templates \citep{Nevsetvril2009,Taxi_hom}. In some sense, this dichotomy defines a simple yet robust frontier between a ``sparse'' and a ``dense'' world.

Examples of nowhere dense classes include classes with bounded expansion, which we now define formally. Let $\mathcal C$ be a graph class. Define 
$$\rdens{d}(\mathcal C)=\sup_{G\in\mathcal C\shm d}\frac{\|G\|}{|G|}\enspace.$$
In the particular case of a single-element class $\{G\}$, $\rdens{d}(G)$ is called the {\em greatest reduced average density} (grad) of $G$ of rank $d$. 
We say $\mathcal C$ has {\em bounded expansion} if there exists a function $f:\bbbn\rightarrow\bbbr$ (called an {\em expansion function}) such that
\begin{equation*}
\forall d\in\bbbn\qquad \rdens{d}(\mathcal C)\leq f(d)\enspace.
\end{equation*}
For example, it is easily seen \cite{Taxi_stoc06} that every graph $G$ with maximum degree at most $D$ satisfies $\rdens{d}(G)<D^{d+1}$. Thus a class of graphs with bounded maximum degree has bounded expansion.

Define $\rdens{}(\mathcal C)=\rdens{\infty}(\mathcal C)$. The graph classes with bounded expansion, where the expansion function is bounded by a constant, are precisely those excluding a fixed minor. Let $h(G)$ be the Hadwiger number of a graph $G$; that is, $K_{h(G)}$ is a minor of $G$ but $K_{h(G)+1}$ is not a minor of $G$. Then \citet{NesOdM-GradI} showed that
\begin{equation}
\eqnlabel{GradConstant}
\half(h(G)-1)\leq\rdens{}(G)\leq \Oh{h(G)\sqrt{\log h(G)}}\enspace.
\end{equation}

\section{Characterisations of Bounded Expansion Classes}
\seclabel{Characterisations}

Several characterisations of bounded expansion classes are known, based on:
\begin{itemize}
	\item special decompositions, namely {\em low tree-depth decompositions} \citep{NesOdM-TreeDepth-EJC06,NesOdM-GradI};
	\item orientations and augmentations, namely {\em transitive fraternal augmentations} \citep{NesOdM-GradI};
	\item vertex orderings, namely {\em generalised weak colouring numbers} \citep{Zhu2008};
	\item edge densities of shallow topological minors \citep{Dvo,Dvorak-EUJC08}.
\end{itemize}
Here we recall this last characterisation and then give two new characterisations.

\subsection{Characterisation by Shallow Topological Minors}
A graph $H$ is a \emph{subdivision} of a graph $G$ if $H$ is obtained by replacing each edge $vw$ of $G$ by a path between $v$ and $w$. The vertices in $H-V(G)$ are called \emph{division vertices}. The vertices in $V(G)$ are called \emph{original vertices}. A subdivision of $G$ with at most $t$ division vertices on each edge of $G$ is called a \emph{$(\leq t)$-subdivision}. The subdivision of $G$ with exactly $t$ division vertices on each edge of $G$ is called the \emph{$t$-subdivision} of $G$. The $1$-subdivision of $G$ is denoted by $G'$. In a $(\leq1)$-subdivision of $G$, if $x$ is the division vertex for some edge $vw$ of $G$, then the path $(v,x,w)$ in $G'$ is called a \emph{transition}.

A {\em shallow topological minor} of a graph $G$ of depth $d$ is a (simple) graph $H$ obtained from a subgraph of $G$ by replacing an edge disjoint family of induced paths of length at most $2d+1$ by single edges (see \figref{topmin}).

\begin{figure}[htb]
$$\xy\xymatrix@C=3cm{
\includegraphics[width=3cm]{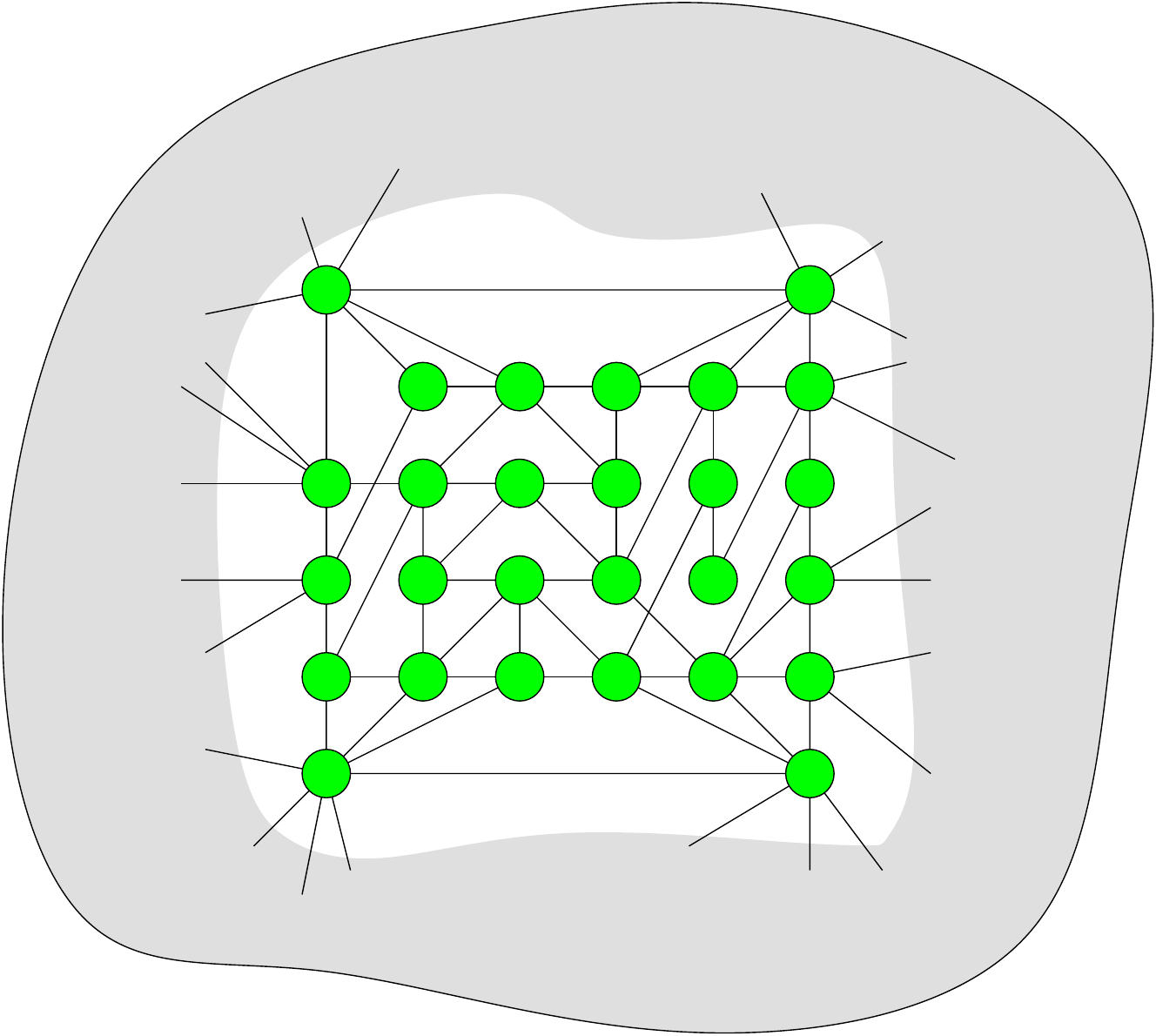}\ar[r]^{subgraph}&
\includegraphics[width=3cm]{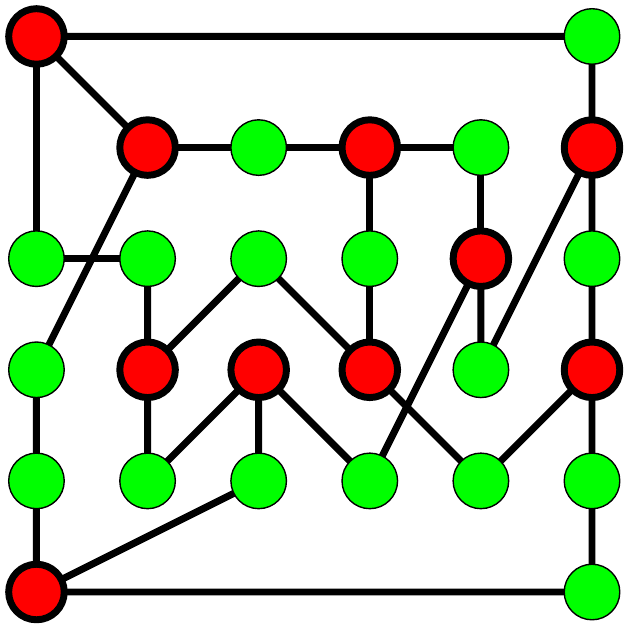}\ar[d]^{\approx}\\
\includegraphics[width=3cm]{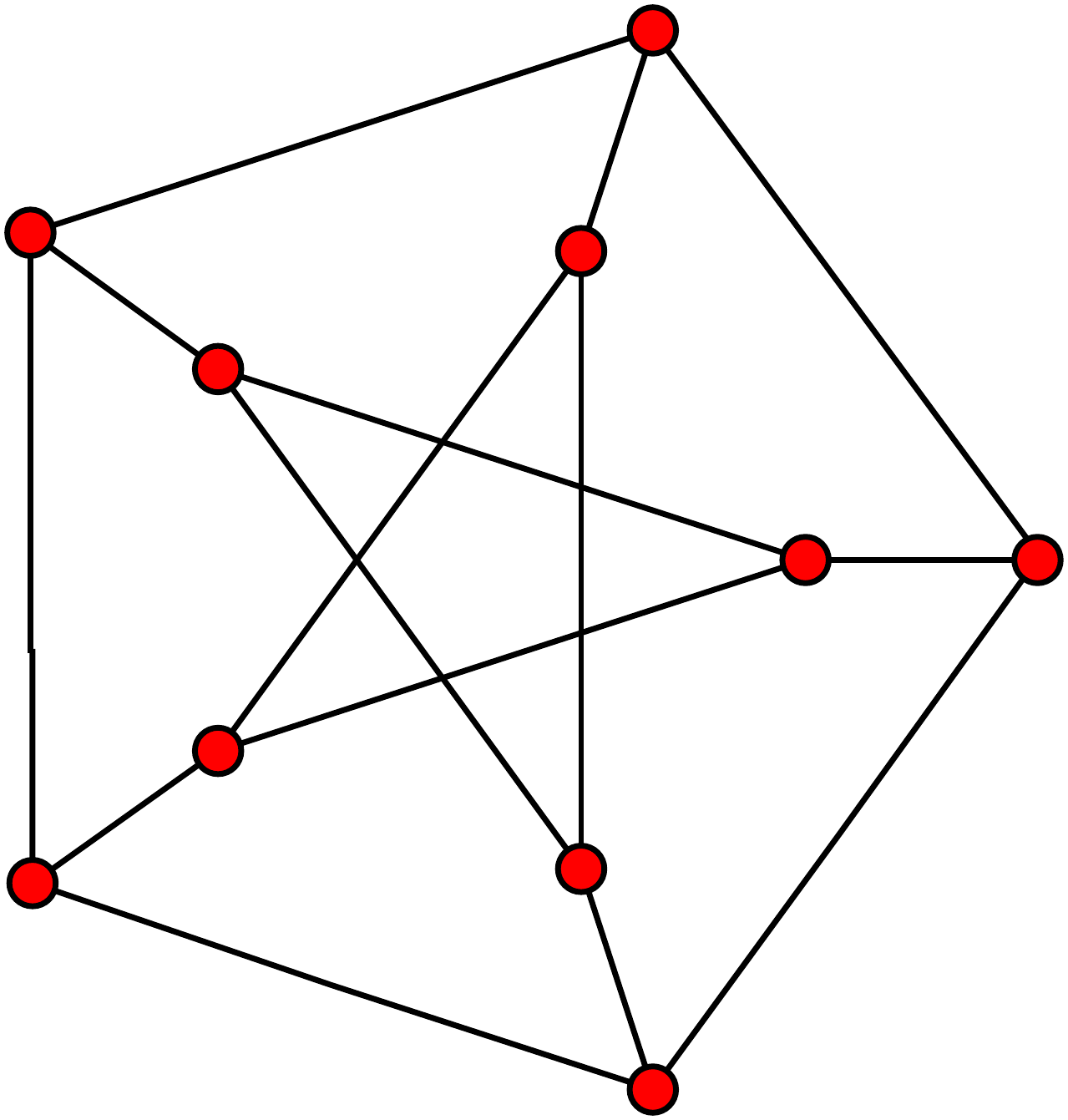}&
\includegraphics[width=3cm]{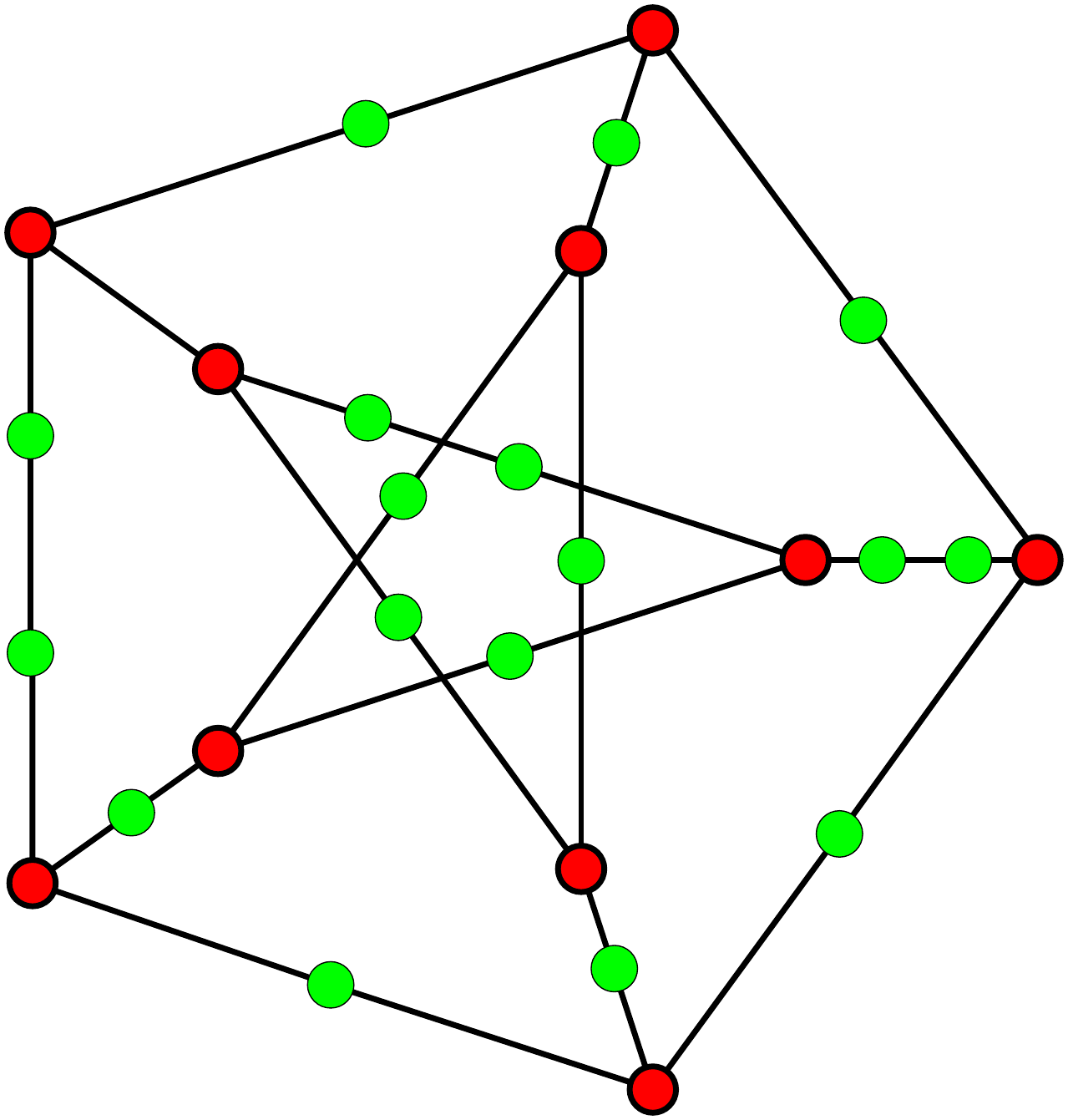}\ar[l]_{\text{short path contraction}}
}\endxy$$
\caption{A Petersen topological minor of depth $1$ in a graph}
\label{fig:topmin}
\end{figure}

For a graph $G$ and $d\in\bbbn$, let $G\shtm d$ denote be the class of graphs that are shallow topological minors of $G$ at depth $d$. As a special case, $G\shtm 0$ is the class of all subgraphs of $G$ (no contractions allowed). Since $G\shtm d$ is contained in $G\shm d$,
$$
\xy\xymatrix@=0pt{
&&G\shm 0&\subseteq&G\shm 1&\subseteq&\dots&G\shm d&\subseteq&\dots&\subseteq C\shm\infty\\
&&\req&&\rsub&&&\rsub&&&\rsub\\
G&\in&G\shtm 0&\subseteq&G\shtm 1&\subseteq&\dots&G\shtm d&\subseteq&\dots&\subseteq C\shtm\infty}\endxy$$
For a class of graphs $\mathcal C$, define
\begin{equation*}
\mathcal C\shtm d=\bigcup_{G\in\mathcal C}G\shtm d\enspace.
\end{equation*}
Hence $\{G\}\shtm d=G\shtm d$ for every graph $G$, and we have the non-decreasing sequence
\begin{equation*}
\mathcal C\shtm 0\subseteq \mathcal C\shtm 1\subseteq \mathcal C\shtm 2\subseteq\dots\subseteq \mathcal C\shtm d\subseteq\dots\subseteq \mathcal C\shtm\infty\enspace.
\end{equation*}
The {\em topological closure} of $\mathcal C$ is the class $\mathcal C\shtm \infty$ of all topological minors of graphs in $\mathcal C$. We say $\mathcal C$ is {\em topologically-closed} if $\mathcal C=\mathcal C\shtm\infty$, and is {\em proper topologically-closed} if it is topologically-closed and does not include all (simple finite) graphs. Define
$$\trdens{d}(\mathcal C)=\sup_{G\in\mathcal C\shtm d}\frac{\|G\|}{|G|}\enspace,$$
and denote $\trdens{\infty}(\mathcal C)$ by $\trdens{}(\mathcal C)$. In the particular case of a single element class $\{G\}$, $\trdens{d}(G)$ is called the {\em topological greatest reduced average density} (top-grad) of $G$ of rank $d$. Obviously, $\rdens{d}(G)$ is an upper bound for $\trdens{d}(G)$. That a polynomial function of $\rdens{d}(G)$ is also a lower bound for $\trdens{d}(G)$ was proved by Zden{\v{e}}k Dvo{\v{r}}{\'a}k in his Ph.D.\ thesis:

\begin{theorem}[\citep{Dvo,Dvorak-EUJC08}]
\thmlabel{Dvorak}
Let $G$ be a graph and $d,\delta\in\bbbn^+$. If $\rdens{d}(G)\geq 4(4\delta)^{(d+1)^2}$, then $G$ contains a subgraph that is a $(\leq 2d)$-subdivision of a graph with minimum degree $\delta$.
\end{theorem}


\begin{corollary}
\corlabel{Dvorak}
For every graph $G$ and $d\in\bbbn$,
\begin{equation*}
\trdens{d}(G)\leq \rdens{d}(G)\leq 4(4\trdens{d}(G))^{(d+1)^2}\enspace.
\end{equation*}
\end{corollary}

If follows that a class $\mathcal C$ has bounded expansion if and only if there is a function $f:\bbbn\rightarrow\bbbn$ such that $\trdens{d}(G)\leq f(d)$ for every graph $G\in\mathcal C$. This alternative characterisation will be particularly useful in this paper.
For example, every graph $G$ with maximum degree at most $D$ satisfies $\trdens{}(G)\leq D/2$ (as the bound on the maximum degree obviously holds for every topological minor of $G$). 

\subsection{Characterisation by Topological Parameters}

Here we introduce the first of our new characterisations of bounded expansion classes. A \emph{graph parameter} is a function $\alpha$ for which $\alpha(G)$ is a non-negative real number for every graph $G$. Note that all the graph parameters that we shall study are isomorphism-invariant. Examples include minimum degree, average degree, maximum degree, connectivity, chromatic number, treewidth, etc. If $\alpha$ and $\beta$ are graph parameters, then $\alpha$ is \emph{bounded by} $\beta$ if for some function $f$,  $\alpha(G)\leq f(\beta(G))$ for every graph $G$. 


\citet{Dujmovic2005} defined a graph parameter $\alpha$ to be {\em topological} if for some function $f$, for every graph $G$, $\alpha(G)\leq f(\alpha(G'))$ and $\alpha(G')\leq f(\alpha(G))$ where $G'$ is the $1$-subdivision of $G$ (the graph obtained from $G$ by subdividing each edge once). For instance, tree-width and genus are topological, but chromatic number is not. A graph parameter $\alpha$ is {\em strongly topological} if for some function $f$, for every graph $G$ and every $\leq 1$-subdivision $H$ of $G$, $\alpha(G)\leq f(\alpha(H))$ and $\alpha(H)\leq f(\alpha(G))$.
The graph parameter $\alpha$ is {\em monotone} (respectively, {\em hereditary}) if $\alpha(H)\leq\alpha(G)$ for every subgraph (respectively, every induced subgraph) $H$ of $G$, and $\alpha$ is {\em degree-bound} if for some function $f$, every graph $G$ has a vertex of degree at most $f(\alpha(G))$. Notice that in such a case, $f$ may be chosen non-decreasing. A graph parameter $\alpha$ is {\em unbounded} if for every integer $N$ there exists a graph $G$ such that $\alpha(G)>N$. 

\begin{lemma}
\lemlabel{BEparam}
\lemlabel{TopoExp}
A class $\mathcal C$ has bounded expansion if and only if there exists a strongly topological, monotone, degree bound graph parameter $\alpha$ and a constant $c$ such that $\mathcal C\subseteq\{G: \alpha(G)\leq c\}$.
\end{lemma}

\begin{proof}
Assume $\mathcal C$ has bounded expansion, and let $f(r)=\trdens{r}(\mathcal C)$.
If $f$ is bounded, then define $\alpha(G)=\trdens{\infty}(G)$.
Otherwise, define $\alpha(G)$ to be the minimum $\lambda\geq 1$ such that
 $\trdens{r}(G)\leq f(\lambda(r+1))$  for every $r\geq 0$.
Let $G$ be a graph and let $H$ be a $\leq 1$-subdivision of $G$. Then $\trdens{r}(H)\leq\trdens{r}(G)\leq f(\alpha(G)(r+1))$ and $\trdens{r}(G)\leq\trdens{2r+1}(H)\leq f(\alpha(H)(2r+2))=f(2\alpha(H)(r+1))$.
It follows that $\alpha(H)\leq\alpha(G)\leq 2\alpha(H)$. This proves that $\alpha$ is strongly topological.
If $H$ is a subgraph of $G$ then $\trdens{r}(H)\leq\trdens{r}(G)$, hence $\alpha(H)\leq\alpha(G)$. Also, every graph $G$ has a vertex of degree at most $2\trdens{0}(G)\leq 2f(\alpha(G))$, hence $\alpha$ is also degree-bound. Notice that  $\mathcal C$ obviously is a subset of $\{G: \alpha(G)\leq 1\}$.

Now assume that $\alpha$ is a strongly topological, monotone, and degree-bound parameter. Let $\mathcal C=\{G:\alpha(G)\leq c\}$ for some constant $c$. Let $r$ be an integer. Let $G\in\mathcal C$.
For some $H\in G\shtm r$, we have $\trdens{0}(H)=\trdens{r}(G)$.
Let $S$ be a $\leq r$-subdivision of $H$ isomorphic to a subgraph of $G$. 
Let $p=\lceil\log_2(2r)\rceil$. There is a sequence
$H=H_0,H_1,\dots, H_p=S$ such that $H_{i+1}$ is a $\leq 1$-subdivision of $H_i$, for each $i\in\{0,\dots,p-1\}$. By induction, $\alpha(H)\leq f^p(\alpha(S)$ where $f^p$ is $f$ iterated $p$ times. Since $f$ is non-decreasing, $\alpha(H)\leq f^p(c)$. Since $\alpha$ is degree bound and hereditary, $\trdens{r}(G)=\trdens{0}(\mathcal H)$ is at most some $D=D(f^p(c))$. It follows that $\mathcal C$ has bounded expansion.
\end{proof}

\subsection{Characterisation by Controlling Dense Parts}

Here we introduce the second of our new characterisations of bounded expansion classes.

\begin{lemma}
\lemlabel{trdens0}
 For every graph $G$ and every integer $r$, if $\trdens{r}(G)> 2$ then
\begin{equation}
 \trdens{0}(G)> 1+\frac{1}{2r+1}\enspace.
\end{equation}
\end{lemma}

\begin{proof}
For some $H\in G\shtm r$, we have $\trdens{r}(G)=\trdens{0}(H)$.
Let $G'$ be a $\leq r$-subdivision of $H$ that is a subgraph of $G$. 
Let $\overline{r}$ be the average number of subdivision vertices of $G'$ per branch. 
Then $\card{G'}=\card{H}+\overline{r}\|H\|$ and $\|G'\|=\|H\|+\overline{r}\|H\|$. Hence
\begin{equation*}
 \trdens{0}(G)\geq \frac{\|G'\|}{\card{G'}}
=\frac{\|H\|+\overline{r}\|H\|}{\card{H}+\overline{r}\|H\|}
=\frac{1+\overline{r}}{1/\trdens{r}(H)+\overline{r}}
>1+\frac{1}{2r+1}\enspace.
\end{equation*}
\end{proof}

This property may be efficiently used in conjunction with the following alternative characterisation of classes with bounded expansion, which may be useful for classes that are neither addable (that is, closed under disjoint unions) nor hereditary.

\begin{lemma}
\lemlabel{charac2}
Let class $\mathcal C$ be a class of graphs. Then $\mathcal C$ has bounded expansion if, and only if, 
there exists functions $F_{\rm ord},F_{\rm deg}, F_{\nabla}, F_{\rm prop}:\bbbr\rightarrow\bbbr$ such that the following 
two conditions hold:
\begin{itemize}
\item$\displaystyle
\forall\epsilon>0,\quad\forall G\in\mathcal C,\quad |G|>F_{\rm ord}(\epsilon)\Longrightarrow \frac{\card{\{v\in G: d(v)\geq F_{\rm deg}(\epsilon)\}}}{\card{G}}\leq\epsilon\\
$
\item$\displaystyle
\forall r\in\bbbn,\quad\forall H\subseteq G\in\mathcal C,\quad \trdens{r}(H)>F_{\nabla}(r)\Longrightarrow \card{H}>
F_{\rm prop}(r)\card{G}
$
\end{itemize}
\end{lemma}

\begin{proof}
Assume $\mathcal C$ has bounded expansion. Then the average degree of graphs in $\mathcal C$ is bounded by 
 $2\rdens{0}(\mathcal C)$. Hence, for every $G\in\mathcal C$ and every integer $k\geq 1$,
\begin{align*}
2\rdens{0}(\mathcal C)
\geq\frac{\sum_{i\geq 1}i\ \card{\{v\in G: d(v)=i\}}}{\card{G}}
&=\frac{\sum_{i\geq 1} \card{\{v\in G: d(v)\geq i\}}}{\card{G}}\\
&\geq k\frac{\card{\{v\in G: d(v)\geq k\}}}{\card{G}}\enspace.
\end{align*}
Hence $\frac{\card{\{v\in G: d(v)\geq k\}}}{\card{G}}\leq \frac{2\rdens{0}(\mathcal C)}{k}$. Thus
$F_{\rm ord}(\epsilon)=0$ and $F_{\rm deg}(\epsilon)=\left\lceil\frac{2\rdens{0}(\mathcal C)}{\epsilon}\right\rceil$ suffice.
The second property is straightforward: put $F_{\nabla}(r)=\trdens{r}(\mathcal C)$ and $F_{\rm prop}(r)=1$.

Now assume that the two conditions hold. 
Fix $r$. Let $G\in\mathcal C$ and let $S$ be a subset of vertices of $G$ of cardinality $t\leq \frac{F_{\rm prop}(r)}{(r+1)F_{\nabla}(r)} n$. Let $F_r(S)$ denote a vertex subset formed by adding paths of length at most $r+1$ with interior vertices in $V\setminus S$ and endpoints in $S$ (not yet linked by a path), one by one until no path of length at most $r+1$ has interior vertices in $V\setminus S$ and endpoints in $S$.
Then $\card{F(S)}\leq (r+1)F_{\nabla}(r)t$.
Suppose not, and consider the set $T$ of the first $(r+1)F_{\nabla}(r)t\leq F_{\rm prop}(r) n$ vertices of $F(S)$. 
By definition the subgraph of $G$ induced by $T$ contains a $\leq r$-subdivision of a graph $H$ of order $t$ and size at least $\frac{\card{T\setminus S}}{r}=F_{\nabla}(r)t$. It follows that $\trdens{r}(G[T])\geq F_{\nabla}(r)$ hence $\card{T}>F_{\rm prop}(r) n$, a contradiction.

Let $D_0=F_{\rm deg}(\frac{F_{\rm prop}(r)}{(r+1)F_{\nabla}(r)})$.
Then for sufficiently big graphs $G$ (of order greater than $N=F_{\rm ord}(\frac{F_{\rm prop}(r)}{(r+1)F_{\nabla}(r)})$), $\frac{\card{\{v\in G: d(v)\geq D_0\}}}{\card{G}}< \frac{F_{\rm prop}(r)}{(r+1)F_{\nabla}(r)}$. 
Let $D=\max(D_0,(r+1)F_{\nabla}(r))$. 
Now assume that there exists in $G$ a $\leq r$-subdivision $G'$ of a graph $H$ with minimum degree at least $D$. 
As $\card{H}$ is the number of vertices of $G'$ having degree at least $D$, we infer that $\card{H}\leq \frac{F_{\rm prop}(r)}{(r+1)F_{\nabla}(r)}n$. It follows that $\card{G'}\leq (r+1)F_{\nabla}(r)\card{H}$ hence $D\leq\|H\|/\card{H}<(r+1)F_{\nabla}(r)\leq D$, a contradiction. It follows that $\trdens{r}(G)< 2D$. Hence, for every graph $G\in\mathcal C$ (including those of order at most $N$) we have
$\trdens{r}(G)< 2\max(F_{\rm ord}(\frac{F_{\rm prop}(r)}{(r+1)F_{\nabla}(r)}),F_{\rm deg}(\frac{F_{\rm prop}(r)}{(r+1)F_{\nabla}(r)}),(r+1)F_{\nabla}(r))$.
\end{proof}

\section{Random Graphs (Erd\H os-R\'enyi model)}
\seclabel{Random}

The $G(n,p)$ model of random graphs was introduced by \citet{Gilbert1959} and \citet{ErdHos1960}; see  \citep{Bollobas2001}. In this model, a graph with $n$ vertices is built, where each edge appears independently with probability $p$. It is frequently considered that $p$ may be a function of $n$, hence the notation $G(n,p(n))$.

The order of the largest complete (topological) minor in $G(n,p/n)$ is well-studied. It is known since the work of \citet{Luczak1994} that random graphs $G(n,p(n))$ with 
$p(n)-1/n\ll n^{-4/3}$ are asymptotically almost surely (henceforth abbreviated, \emph{a.a.s.}) planar, whereas those with $p(n)-1/n\gg n^{-4/3}$  a.a.s.\ contain unbounded clique minors.
\citet{FKO-RSA08} proved that for every $c>1$ there exists a constant $\delta(c)$ such that a.a.s.\ the maximum order $h(G(n,c/n))$ of a complete minor of a graph in $G(n,c/n)$ satisfies the inequality $\delta(c)\sqrt{n}\leq h(G(n,c/n))\leq 2\sqrt{cn}$.
Also, \citet{Ajtai1979} proved that as long as the expected degree $(n-1)p$ is at least $1+\epsilon$ and is $o(\sqrt{n})$, then a.a.s.\ the order of the largest complete topological minor of $G(n,p)$ is almost as large as the maximum degree, which is $\Theta(\log n/\log \log n)$.

However, it is known that the number of short cycles of $G(n,c/n)$ is bounded. Precisely, the expected number of cycles of length $t$ in $G(n,c/n)$ is at most $(e^2c/2)^t$. It follows that the expected value ${\mathrm E}(\omega(G\shtm d))$ of the clique size of a shallow topological minor of $G$ at depth $d$ is bounded by approximately $(Ac)^{2d}$ (for some constant $A>0$).


\citet{FoxSudakov-EuJC} proved that $G(n,d/n)$ is a.a.s.\ $(16d,16d)$-degenerate, where a graph $H$ is said to be {\em $(d,\Delta)$-degenerate} if there exists an ordering $v_1,\dots,v_n$ of its vertices such that for each $v_i$, there are at most $d$ vertices $v_j$ adjacent to $v_i$ with $j<i$, and there are at most $\Delta$ subsets $S\subset N(v_j)\cap\{v_1,\dots,v_i\}$ for some neighbour $v_j$ of $v_i$ with $j>i$, where the {\em neighbourhood} $N(v_j)$ is the set of vertices that are adjacent to $v_j$.  
We modify their proof in order to estimate the top-grads of $G(n,d/n)$. The proof is based on the characterisation of bounded expansion given in \lemref{charac2}. We first prove that graphs in $G(n,d/n)$ a.a.s.\ have a small proportion of vertices have sufficiently large degree, and then that subgraphs having sufficiently dense sparse topological minors must span some positive fraction of the vertex set of the whole graph. Thanks to \lemref{trdens0}, this last property will follow from the following two facts: 
\begin{itemize}
	\item As a random graph with edge probability $d/n$ has a bounded number of short cycles, if one of its subgraphs is a $\leq r$-subdivision of a sufficiently dense graph it should a.a.s.\ span at least
	some positive fraction $F_{\rm prop}(r)$ of the vertices (\twolemref{sparse}{sparse2});
	\item For every $\epsilon>0$, the proportion of vertices in a random graph with edge probability $d/n$ that have sufficiently large degree ($>F_{\rm deg}(\epsilon)$) is a.a.s.\ less than $\epsilon$
	(\lemref{smalldeg}). 
\end{itemize}

\begin{lemma}
\lemlabel{sparse}
Let $\epsilon>0$. A.a.s.\ every subgraph $G'$ of $G(n,d/n)$ with $t\leq (4d)^{1+1/\epsilon} n$ vertices satisfies $\trdens{0}(G')\leq 1+\epsilon$.
\end{lemma}
\begin{proof}
It is sufficient to prove that almost surely every subgraph $G'$ of $G(n,d/n)$ with $t\leq 4^{1+1/\epsilon} n$ vertices satisfies $\|G'\|/\card{G'}\leq 1+\epsilon$.
Let $G'$ be an induced subgraph of $G$ of order $t$ with $t\leq 4^{1+1/\epsilon} n$. The probability that $G'$ has size at least $m=(1+\epsilon)t$ is at most $\binom{\binom{t}{2}}{m}(d/n)^m$. Therefore, by the union bound, the probability that $G$ has an induced subgraph of order $t$ with size at least $m=(1+\epsilon)t$ is
\begin{align*}
\binom{n}{t}\binom{\binom{t}{2}}{m}(d/n)^m
&\leq \biggl(\frac{en}{t}\biggr)^t\biggl(\frac{et^2}{2m}\biggr)^m\biggl(\frac{d}{n}\biggr)^m\\
&=e^t \biggl(\frac{e}{2(1+\epsilon)}\biggr)^{(1+\epsilon)t}\biggl(\frac{n}{t}\biggr)^t\biggl(\frac{dt}{n}\biggr)^{(1+\epsilon)t}\\
&= \biggl(\frac{e^{2+\epsilon}}{(2+2\epsilon)^{1+\epsilon}}\biggr)^t\biggl(\frac{d^{1+1/\epsilon}\ t}{n}\biggr)^{\epsilon t}\\
&< 4^t\biggl(\frac{d^{1+1/\epsilon}\ t}{n}\biggr)^{\epsilon t}.
\end{align*}
Summing over all $t\leq (4d)^{-(1+1/\epsilon)}n$, one easily checks that the probability that $G$ has an induced subgraph $G'$ of order at most $(4d)^{-(1+1/\epsilon)}n$ such that $\|G\|/\card{G'}\geq 1+\epsilon$ is $o(1)$, completing the proof.
\end{proof}

\twolemref{trdens0}{sparse} imply:

\begin{lemma}
\lemlabel{sparse2}
Let $r\in\bbbn$. A.a.s.\ every subgraph $G'$ of $G(n,d/n)$ with $t\leq (4d)^{-(1+1/(2r+1))} n$ vertices satisfies $\trdens{r}(G')\leq 2$. That is,
\begin{equation*}
\forall r\in\bbbn,\quad a.a.s.\quad \forall H\subseteq G(n,d/n),\quad \trdens{r}(H)>2\Longrightarrow \card{H}>
(4d)^{-(1+\frac{1}{2r+1})}\card{G}\enspace.
\end{equation*}
\end{lemma}


\begin{lemma}
\lemlabel{smalldeg}
Let $\alpha>1$ and let $c_\alpha=4e\alpha^{-4\alpha d}$. 
A.a.s.\ there are at most $c_\alpha n$ vertices of $G(n,d/n)$ with degree greater than $8\alpha d$.
\end{lemma}
\begin{proof}
Let $A$ be the subset of $s=c_\alpha n$ vertices of largest degree in $G=G(n,d/n)$, and let $D$ be the minimum degree of vertices in $A$. Thus there are at least $sD/2$ edges that have at least one endpoint in $A$. Consider a random subset $A'$ of $A$ with size $\card{A}/2$. Every edge that has an endpoint in $A$ has probability at least $\tfrac{1}{2}$ of having exactly one endpoint in $A'$. So there is a subset $A'\subset A$ of size $\card{A}/2$ such that the number $m$ of edges between $A'$ and $V(G)\setminus A'$ satisfies $m\geq sD/4=\card{A'}D/2$.

We now give an upper bound on the probability that $D\geq 8\alpha d$. Each set $A'$ of $\tfrac{s}{2}$ vertices in $G=G(n,d/n)$ has probability at most 
$$\binom{\frac{s}{2}(n-\frac{s}{2})}{m}(d/n)^m\leq\biggl(\frac{esn}{2m}\biggr)^m(d/n)^m\leq\biggl(\frac{2sd}{m}\biggr)^m\leq
\biggl(\frac{8d}{D}\biggr)^m\leq \alpha^{-2\alpha ds}$$
of having at least $m\geq (s/2)(8\alpha d)/2=2\alpha ds$ edges between $A'$ and $V(G)\setminus A'$. Therefore the probability that there is a set $A'$ of $s/2$ vertices in $G$ that has at least $2\alpha sd$ edges between $A'$ and $V(G)\setminus A'$ is at most
$$\binom{n}{s/2}\alpha^{-2\alpha ds}<\biggl(\frac{2en}{s}\biggr)^{s/2}\alpha^{-2\alpha ds}\leq \left(\frac{(2e\alpha^{-4\alpha d})n}{s}\right)^{s/2}=o(1),$$
completing the proof.
\end{proof}

\begin{theorem}
For every $p>0$ there exists a class $\mathcal R_p$ with bounded expansion such that
$G(n,p/n)$ asymptotically almost surely belongs to $\mathcal R_p$.
\end{theorem}


\section{Crossing Number}
\seclabel{CrossingNumber}

For a graph $G$, let \CR{G} denote the \emph{crossing number} of $G$, defined to be the minimum number of crossings in a drawing of $G$ in the plane; see the surveys \citep{PachToth-JCTB00,Szekely-DM04}. It is easily seen that $\CR{H}=\CR{G}$ for every subdivision $H$ of $G$. Thus crossing number is strongly topological. The following ``crossing lemma'', independently due to \citet{Leighton83} and \citet{Ajtai82}, implies that crossing number is degree-bound.

\begin{lemma}[\citep{Leighton83, Ajtai82,ProofsFromTheBook}]
\lemlabel{CrossingLemma}
If $\|G\|\geq 4\,|G|$ then $\CR{G}\geq\frac{\|G\|^3}{64\,|G|^2}$.
\end{lemma}

\twolemref{BEparam}{CrossingLemma} imply that a class of graphs with bounded crossing number has bounded expansion. In fact, since every graph $G$ has orientable genus at most $\CR{G}$ (simply introduce one handle for each crossing), any class with bounded crossing number is included in a minor-closed class. In particular,
$$\CR{G}\geq\textsf{genus}(G)\geq\textsf{genus}(K_{h(G)})=\CEIL{\frac{(h(G)-3)(h(G)-4)}{12}}\enspace,$$
implying $h(G)\leq\Oh{\sqrt{\CR{G}}}$ and $\rdens{}(G)\leq\Oh{\sqrt{\CR{G}\log\CR{G}}}$
by \eqnref{GradConstant}.

%
%
%

The following theorem says that graphs with linear crossing number (in some sense) are contained in a topologically-closed class, and thus have bounded expansion. Let $G_{\geq 3}$ denote the subgraph of $G$ induced by the vertices of $G$  that have degree at least $3$.

\begin{theorem}
Let $c\geq 1$ be a constant. Let $\mathcal C_c$ be the class of graphs $G$ such that $\CR{H}\leq c |H_{\geq 3}|$ for every subgraph $H$ of $G$. Then $\mathcal C_c$ is contained in a topologically-closed class of graphs. Precisely $\trdens{}(\mathcal C_c)\leq 4c^{1/3}$.
\end{theorem}

\begin{proof}
Let $G\in\mathcal C_c$ and let $H$ be a topological minor of $G$ such that $\|H\|/|H|=\trdens{}(G)$. Let $S\subseteq G$ be a witness subdivision of $H$ in $G$. We prove that $\|H\|\leq 4c^{1/3}|H|$ by contradiction. Were it false, then $\|H\|>4c^{1/3}|H|$ and by \lemref{CrossingLemma},
\begin{equation*}
\frac{\|H\|^3}{64|H|^2}\leq \CR{H}=\CR{S}\leq \CR{S_{\geq 3}}=c|H|\enspace.
\end{equation*}
Thus $\|H\|^3<64c|H|^3$, a contradiction. Hence $\trdens{}(G)\leq 4c^{1/3}$ for every $G\in\mathcal C_c$.
\end{proof}

Consider the class of graphs that admit drawings with at most one crossing per edge. Obviously this includes large subdivisions of arbitrarily large complete graphs. Thus this class is not contained in a proper topologically-closed class. However, it does have bounded expansion.

\begin{theorem}
Let $c\geq 1$ be a constant. The class of graphs $G$ that admit a drawing with at most $c$ crossings per edge has bounded expansion. Precisely, $\trdens{d}(G)\in\mathcal O(\sqrt{cd})$.
\end{theorem}

\begin{proof}
Assume  $G$ admits a drawing with at most $c$ crossings per edge. 
Consider a subgraph $H$ of $G$ that is a $(\leq 2d)$-subdivision of a graph $X$. 
So $X$ has a drawing with at most $c(2d+1)$ crossings per edge.
\citet{PachToth97} proved that if an $n$-vertex graph has a drawing with at most $k$ crossings per edge, then it has at most $4.108\sqrt{k}n$ edges.
Thus $\|X\|\leq 4.108\sqrt{c(2d+1)}\,|X|$ hence $\trdens{d}(G)\leq 4.108\sqrt{c(2d+1)}$. 
\end{proof}


\section{Queue and Stack Layouts}

A graph $G$ is \emph{ordered} if $V(G)=\{1,2,\dots,|G|\}$. Let $G$ be an ordered graph. Let $\ell(e)$ and $r(e)$ denote the endpoints of each edge $e\in E(G)$ such that $\ell(e)\leq r(e)$. Two edges $e$ and $f$ are \emph{nested} and $f$ is \emph{nested inside} $e$ if $\ell(e)<\ell(f)$ and $r(f)<r(e)$. Two edges $e$ and $f$ \emph{cross} if $\ell(e)<\ell(f)<r(e)<r(f)$.

An ordered graph is a \emph{queue} if no two edges are nested. An ordered graph is a \emph{stack} if no two edges cross. Observe that the left and right endpoints of the edges in a queue are in first-in-first-out order, and are in last-in-first-out order in a stack---hence the names `queue' and `stack'. 

Let $G$ be an ordered graph. $G$ is a $k$-\emph{queue} if there is a partition $\{E_1,E_2,\dots,E_k\}$ of $E(G)$ such that each $G[E_i]$ is a queue. 
$G$ is a $k$-\emph{stack} if there is a partition $\{E_1,E_2,\dots,E_k\}$ of $E(G)$ such that each $G[E_i]$ is a stack. 

Let $G$ be an (unordered) graph. A \emph{$k$-queue layout} of $G$ is a $k$-queue that is isomorphic to $G$. A \emph{$k$-stack layout} of $G$ is a $k$-stack that is isomorphic to $G$. A $k$-stack layout is often called a \emph{$k$-page book embedding}. The \emph{queue-number} of $G$ is the minimum integer $k$ such that $G$ has a $k$-queue layout. The \emph{stack-number} of $G$ is the minimum integer $k$ such that $G$ has a $k$-queue layout. 

Stack layouts are more commonly called \emph{book embeddings}, and stack-number has been called \emph{book-thickness}, \emph{fixed outer-thickness}, and \emph{page-number}. See \citep{DujWoo-DMTCS04} for references and applications of queue and stack layouts.

\citet{BK79} proved that a graph has stack number $1$ if and only if it is outerplanar, and it has stack number at most $2$ if and only if it is a subgraph of a Hamiltonian planar graph  (see \figref{sn}). Thus every $4$-connected planar graph has stack number at most $2$.  \citet{Yannakakis89} proved that every planar graph has stack number at most $4$. In fact, every proper minor-closed class has bounded stack-number \citep{Blankenship-PhD03}. On the other hand, even though stack and queue layouts appear to be dual, it is unknown whether planar graphs have bounded queue-number \citep{HLR-SJDM92,HR-SJC92}, and more generally, it is unknown whether queue-number is bounded by stack-number \citep{DujWoo-DMTCS05}. \citet{DujWoo-DMTCS05} proved that planar graphs have bounded queue-number if and only if $2$-stack graphs have bounded queue-number, and that  queue-number is bounded by stack-number if and only if $3$-stack graphs have bounded queue-number. The largest class of graphs for which queue-number is known to be bounded is the class of graphs with bounded tree-width \citep{DMW-SJC05}. 

\begin{figure}[ht]
\includegraphics[width=\textwidth]{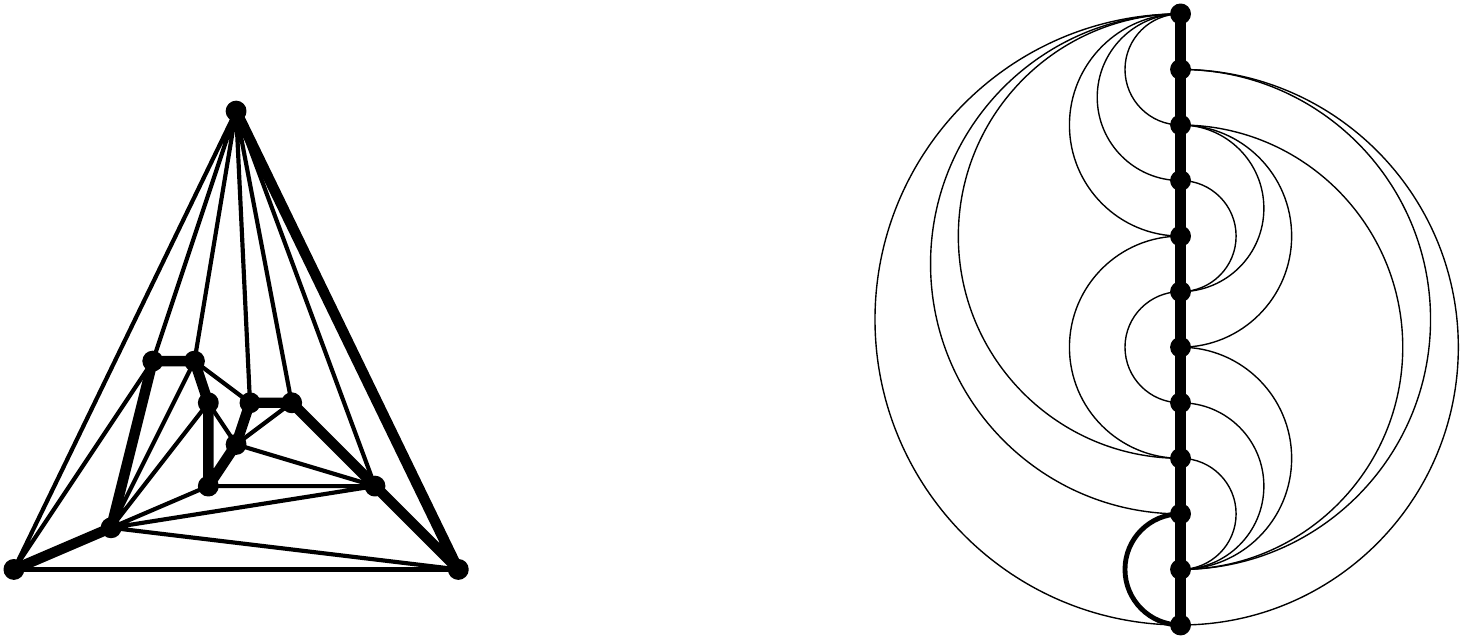}
\caption{Every $4$-connected planar graph has stack number at most $2$ (since it is Hamiltonian)}
\label{fig:sn}
\end{figure}

\begin{figure}[ht]
\includegraphics[width=\textwidth]{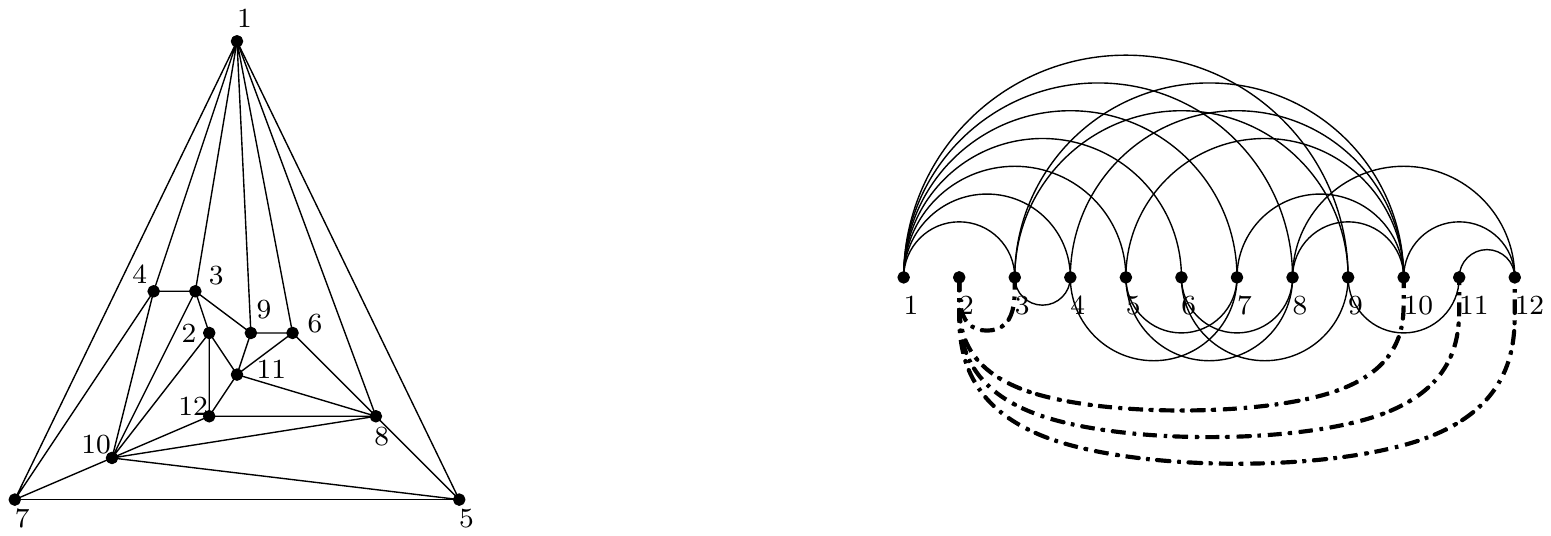}
\caption{A $3$-queue layout of a given planar graph.}
\label{fig:qn}
\end{figure}

In the following two sections, we prove that graphs of bounded queue-number or bounded stack-number have bounded expansion. The closest previous result in this direction is that graphs of bounded queue-number or bounded stack-number have bounded acyclic chromatic number. In particular, \citet{DPW-DMTCS04} proved that every $k$-queue graph has acyclic chromatic number at most $4k\cdot 4^{k(2k-1)(4k-1)}$, and every $k$-stack graph has acyclic chromatic number at most $80^{k(2k-1)}$.

\section{Queue Number}
\seclabel{Queue}

Every $1$-queue graph is planar \citep{DPW-DMTCS04,HR-SJC92}. However, the class of $2$-queue graphs is not contained in a proper topologically-closed class since every graph has a $2$-queue subdivision, as proved by \citet{DujWoo-DMTCS05}.  Moreover, the bound on the number of division vertices per edge is related to the queue-number of the original graph.

\begin{theorem}[\citep{DujWoo-DMTCS05}]
\thmlabel{kQueueSubdiv}
For all $k\geq2$, every graph $G$ has a $k$-queue subdivision with at most $c\log_{k}\qn{G}$ division vertices per edge, for some absolute constant $c$.
\end{theorem}

Conversely, the same authors proved that queue-number is strongly topological.

\begin{lemma}[\citep{DujWoo-DMTCS05}]
\lemlabel{SubdivQueue}
If some $(\leq t)$-subdivision of a graph $G$ has a $k$-queue layout, then  $\qn{G}\leq \half(2k+2)^{2t}-1$, and if $t=1$ then $\qn{G}\leq 2k(k+1)$.
\end{lemma}

Also, queue-number is degree bound:

\begin{lemma}[\citep{HR-SJC92,Pemmaraju-PhD,DujWoo-DMTCS04}]
\lemlabel{SizeQueue}
Every $k$-queue graph has average degree less than $4k$.
\end{lemma}

It now follows that:
 
\begin{theorem}
\thmlabel{QueueExpansion}
Graphs of bounded queue-number have bounded expansion. In particular
\begin{equation*}
\trdens{d}(G)< 4(8(2k+2)^{4d})^{(d+1)^2}
\end{equation*}
for every $k$-queue graph $G$.
\end{theorem}

\begin{proof}
Consider a subgraph $H$ of $G$ that is a $(\leq2d)$-subdivision of a graph $X$.
Thus $\qn{H}\leq k$, and $\qn{X}<\half(2k+2)^{4d}$ by \lemref{SubdivQueue}.
Thus the average degree of $X$ is less than $\delta:=2(2k+2)^{4d}$ by \lemref{SizeQueue}.
Hence $\rdens{d}(G)\leq 4(4\delta)^{(d+1)^2}= 4(8(2k+2)^{4d})^{(d+1)^2}$ by \thmref{Dvorak}.
\end{proof}


Note that there is an exponential lower bound on \trdens{d}\ for graphs of bounded queue-number. Fix integers $k\geq2$ and $d\geq1$. Let $G$ be the graph obtained from $K_n$ by subdividing each edge $2d$ times, where $n=k^d$. \citet{DujWoo-DMTCS05} constructed a $k$-queue layout of $G$. Observe that $\trdens{d}(G)\sim n=k^d$. 

We now set out to give a direct proof of \thmref{QueueExpansion} that does not rely on Dvo{\v{r}}{\'a}k's characterisation (\thmref{Dvorak}). 

Consider a $k$-queue layout of a graph $G$. For each edge $vw$ of $G$, let $q(vw)\in\{1,2,\dots,k\}$ be the queue containing $vw$. For each ordered pair $(v,w)$ of adjacent vertices in $G$, let 
$$Q(v,w):=
\begin{cases}
q(vw)	& \text{ if }v<w,\\
-q(wv)	& \text{ if }w<v.
\end{cases}$$
Note that $Q(v,w)$ has at most $2k$ possible values.

\begin{lemma}
\lemlabel{Edges} 
Let $G$ be a graph with a $k$-queue layout. 
Let $vw$ and $xy$ be disjoint edges of $G$ such that $Q(v,w)=Q(x,y)$.
Then $v<x$ if and only if $w<y$.
\end{lemma}

\begin{proof}
Without loss of generality, $v<w$ and $x<y$ since $|Q(v,w)|=|Q(x,y)|$.

Say $v<x$. If $y<w$ then $v<x<y<w$. Thus $xy$ is nested inside $vw$, which is a contradiction since $q(vw)=q(xy)$. Hence $w<y$.

Say $w<y$. If $x<v$ then $x<v<w<y$. Thus $vw$ is nested inside $xy$, which is a contradiction since $q(xy)=q(vw)$. Hence $v<x$.
\end{proof}

By induction, \lemref{Edges} implies:

\begin{lemma} 
\lemlabel{Paths} 
Let $G$ be a graph with a $k$-queue layout. 
Let $(v_1,v_2,\dots,v_r)$ and $(w_1,w_2,\dots,w_r)$ be disjoint paths in $G$, 
such that $Q(v_i,v_{i+1})=Q(w_i,w_{i+1})$ for each $i\in[1,r-1]$.
Then $v_1<w_1$ if and only if $v_r<w_r$.\qed
\end{lemma}

\begin{theorem}
\thmlabel{ContractQueue}
Let $G$ be a graph with a $k$-queue layout.
Let $F$ be a subgraph of $G$ such that each component of $F$ has radius at most $r$.
Let $H$ be obtained from $G$ by contracting each component of $F$.
Then $H$ has a $f_r(k)$-queue layout, where $$f_r(k):=2k\bracket{\frac{(2k)^{r+1}-1}{2k-1}}^2\enspace.$$
\end{theorem}

\begin{proof}
We can assume that $F$ is spanning by allowing $1$-vertex components in $F$.
For each component $X$ of $F$ fix a \emph{centre} vertex $v$ of $X$ at distance at most $r$ from every vertex in $X$. Call $X$ the \emph{$v$-component}. 

Consider a vertex $v'$ of $G$ in the $v$-component of $F$. Fix a shortest path $P(v')=(v=v_0,v_1,\dots,v_s=v')$ between $v$ and $v'$ in $F$. Thus $s\in[0,r]$. Let $$Q(v'):=\big(Q(v_0,v_1),Q(v_1,v_2),\dots,Q(v_{s-1},v_s)\big)\enspace.$$

Consider an edge $v'w'$ of $G$, where $v'$ is in the $v$-component of $F$, $w'$ is in the $w$-component of $F$, and $v\ne w$. Such an edge survives in $H$. Say $v<w$. Colour $v'w'$ by the triple
$$\big(Q(v'),Q(v',w'),Q(w')\big)\enspace.$$
Observe that the number of colours is at most
$$2k\bracket{\sum_{s=0}^r(2k)^s}^2
=2k\bracket{\frac{(2k)^{r+1}-1}{2k-1}}^2\enspace.$$

From the linear order of $G$, contract each component of $F$ into its centre.
That is, the linear order of $H$ is determined by the linear order of the centre vertices in $G$. After contracting there might be parallel edges with different edge colours. Replace parallel edges by a single edge and keep one of the colours.

Consider disjoint monochromatic edges $vw$ and $xy$ of $H$, where $v<w$ and $x<y$. By construction, there are edges $v'w'$ and $x'y'$ of $G$ such that
$v'$ is in the $v$-component, 
$w'$ is in the $w$-component,
$x'$ is in the $x$-component,
$y'$ is in the $y$-component, and 
\begin{align*}
\big(Q(v'),Q(v',w'),Q(w'))=\big(Q(x'),Q(x',y'),Q(y')).
\end{align*}
Thus $|P(v')|=|P(x')|$ and $|P(w')|=|P(y')|$. 
Consider the paths 
\begin{align*}
&(v=v_0,v_1,\dots,v_s=v',w'=w_t,w_{t-1},\dots,w_0=w) \text{ and }\\
&(x=x_0,x_1,\dots,x_s=x',y'=y_t,y_{t-1},\dots,y_0=y),
\end{align*}
Since $Q(v')=Q(x')$, we have $Q(v_i,v_{i+1})=Q(x_i,x_{i+1})$ for each $i\in[0,s-1]$. 
Similarly, since $Q(w')=Q(y')$, we have $Q(w_i,w_{i+1})=Q(y_i,y_{i+1})$ for each $i\in[0,t-1]$. 
Since $Q(v',w')=Q(x',y')$, \lemref{Paths} is applicable to these two paths. Thus $v<x$ if and only if $x<y$. 
Hence $vw$ and $xy$ are not nested. 
Thus the edge colouring of $H$ defines a queue layout. 
\end{proof}

\thmref{ContractQueue} implies \thmref{QueueExpansion} (with a better bound on the expansion function) since by \lemref{SizeQueue}, the graph $H$ in the statement of \thmref{ContractQueue} has bounded density. In particular, if $G$ has a $k$-queue layout then $$\rdens{d}(G)\leq8k\bracket{\frac{(2k)^{d+1}-1}{2k-1}}^2\enspace.$$

\thmref{ContractQueue} basically says that minors and queue layouts are compatible, in the same way that queue layouts are compatible with subdivisions; see \lemref{SubdivQueue}.

\subsection{Jump Number}

Let $P$ be a partially ordered set (that is, a poset). The \emph{Hasse diagram} $H(P)$ of $P$ is the graph whose vertices are the elements of $P$ and whose edges correspond to the \emph{cover relation} of $P$. Here $x$ \emph{covers} $y$ in $P$ if $x>_P y$ and there is no element $z$ of $P$ such that $x>_P z>_P y$. 

A \emph{linear extension} of $P$ is a total order $\preceq$ of $P$ such that $x<_P y$ implies $x\prec y$ for every $x,y\in P$. The \emph{jump number} \jn{P} of $P$ is the minimum number of consecutive elements of a linear extension of $P$ that are not comparable in $P$, where the minimum is taken over all possible linear extensions of $P$.

\citet{HP-SJDM97} proved that the jump number of a poset is at least the queue number of its Hasse diagram minus one; that is, $\qn{H(P)}\leq\jn{P}+1$. It follows that the class of Hasse diagrams of posets having bounded jump-number has bounded queue-number. Thus \thmref{QueueExpansion} implies:

\begin{corollary}
Let $\mathcal P$ be a class of posets with bounded jump number. Then the class $H(\mathcal P)$ of the Hasse diagrams of the posets in $\mathcal P$ has bounded expansion.
\end{corollary}

\section{Stack Number}
\seclabel{Stack}

The class of $3$-stack graphs is not contained in a proper topologically-closed class since every graph has a $3$-stack subdivision \citep{EM99, Miyauchi94, Miyauchi-IEICE05, Atneosen68, BO99}\footnote{The first proof was by \citet{Atneosen68} in 1968, although similar ideas were present in the work of  \citet{Hotz59,Hotz60} on knot projections from 1959.}. Many authors studied bounds on the number of divisions vertices per edge in $3$-stack subdivisions, especially of $K_n$. The most general bounds on the number of division vertices are by \citet{DujWoo-DMTCS05}.

\begin{theorem}[\citep{DujWoo-DMTCS05}]
For all $s\geq3$, every graph $G$ has an $s$-stack subdivision with at most $c\log_{s-1}\min\{\sn{G},\qn{G}\}$ division vertices per edge, for some absolute constant $c$.
\end{theorem}

It is open whether a result like \lemref{SubdivQueue} holds for stack layouts. \citet{BO99} conjectured that such a result exists.

\begin{conjecture}[\citep{BO99}]
\conjlabel{BO}
There is a function $f$ such that $\sn{G}\leq f(\sn{H})$ for every graph $G$ and $(\leq1)$-subdivision $H$ of $G$.
\end{conjecture}

This conjecture would imply that stack-number is topological. This conjecture holds for $G=K_n$ as proved by \citet{BO99}, \citet{EM99}, and \citet{Eppstein-GD02}. The proofs by \citet{BO99} and \citet{Eppstein-GD02} use essentially the same  Ramsey-theoretic argument. 

\citet{EMO99} proved the following bound for the density of graphs having a $\leq t$-subdivision with a $k$-stack layout:

\begin{theorem}[\citep{EMO99}]
\thmlabel{SubdivStack}
Let $G$ be a graph such that some $(\leq t)$-subdivision of $G$ has a $k$-stack layout for some $k\geq3$. Then $$\|G\|\leq\frac{4k(5k-5)^{t+1}}{5k-6}|G|\enspace.$$
\end{theorem}

It follows that graphs with bounded stack number form a class with bounded expansion:

\begin{theorem}
\thmlabel{StackExpansion}
Graphs of bounded stack number have bounded expansion. In particular:
\begin{equation*}
\trdens{r}(G)\leq\frac{4k(5k-5)^{2r+1}}{5k-6}
\end{equation*}
for every $k$-stack graph $G$.
\end{theorem}

\begin{proof}
$(\leq 2)$-stack graphs have bounded expansion since they are planar. Let $G$ be a graph with stack-number $\sn{G}\leq k$ for some $k\geq3$.
Consider a subgraph $H$ of $G$ that is a $(\leq2r)$-subdivision of a graph $X$.
Thus $\sn{H}\leq k$, and by \thmref{SubdivStack},
$$\|X\|\leq\frac{4k(5k-5)^{2r+1}}{5k-6}|X|\enspace.$$
It follows that $\trdens{r}(G)=\frac{\|H\|}{|H|}\leq \frac{4k(5k-5)^{2r+1}}{5k-6}$.
\end{proof}


The following open problem is equivalent to some problems in computational complexity \citep{Kannan85,GKS-Comb89,GKS-JCSS89}.

\begin{open}
Do $3$-stack $n$-vertex graphs have $o(n)$ separators? 
\end{open}

See \citep[Section~8]{NesOdM-GradII} for results relating expansion and separators.

\section{Non-Repetitive Colourings}
\seclabel{NonRep}

Let $f$ be a  colouring of a graph $G$. Then $f$ is \emph{repetitive} on a path $(v_1,\dots,v_{2s})$ in $G$ if $f(v_i)=f(v_{i+s})$ for each $i\in[1,s]$. If $f$ is not repetitive on every path in $G$, then $f$ is \emph{non-repetitive}. Let $\pi(G)$ be the minimum number of colours in a non-repetitive colouring of $G$. 
These notions were introduced by \citet{AGHR-RSA02} and have since been widely studied  \citep{BreakingRhythm, BV-NonRepVertex07, BV-NonRepEdge08, BGKNP-NonRepTree-DM07, BK-AC04, CG-ENDM07, Currie-TCS05, Gryczuk-IJMMS07, Grytczuk-PatternAvoidance, Grytczuk-DM08, KP-DM08,Manin,  MarxSchaefer-DAM09}. The seminal result in this field, proved by \citet{Thue06} in 1906,  (in the above terminology) states that $\pi(P_n)\leq3$. See \citep{Currie-TCS05} for a survey of related results. Note that a non-repetitive colouring is proper ($s=1$). Moreover, a non-repetitive colouring contains no bichromatic $P_4$ ($s=2$), and is thus a star colouring. Hence $\pi(G)\geq\st(G)\geq\chi(G)$.


The main result in this section is that $\pi$ is strongly topological, and that graphs with bounded $\pi$ have bounded expansion. The closest previous result is by \citet{Wood-DMTCS05} who proved that $\st(G')\geq\sqrt{\chi(G)}$ for every graph $G$, and thus $\pi(G')\geq\sqrt{\chi(G)}$. First observe:

\begin{lemma}\lemlabel{NonRepSub}
\textup{(a)} For every $(\leq1)$-subdivision $H$ of a graph $G$, $$\pi(H)\leq\pi(G)+1.$$
\textup{(b)} For every $(\leq2)$-subdivision $H$ of a graph $G$, $$\pi(H)\leq\pi(G)+2.$$
\textup{(c)} For every subdivision $H$ of a graph $G$, $$\pi(H)\leq\pi(G)+3.$$
\end{lemma}

\begin{proof}
First we prove (a). Given a non-repetitive $k$-colouring of $G$, introduce a new colour for each division vertex of $H$. Since this colour does not appear elsewhere, a repetitively coloured path in $H$ defines a repetitively coloured path in $G$. Thus $H$ contains no repetitively coloured path. Part (b) follows by applying (a) twice. 

Now we prove (c). Let $n$ be the maximum number of division vertices on some edge of $G$. \citet{Thue06} proved that $P_n$ has a non-repetitive $3$-colouring $(c_1,c_2,\dots,c_n)$. Arbitrarily orient the edges of $G$. 
Given a non-repetitive $k$-colouring of $G$, 
choose each $c_i$ to be one of three new colours
for each arc $vw$ of $G$ that is subdivided $d$ times, 
colour the division vertices from $v$ to $w$ by $(c_1,c_2,\dots,c_d)$. 
Suppose $H$ has a repetitively coloured path $P$. Since $H-V(G)$ is a collection of disjoint paths, each of which is non-repetitively coloured, $P$ includes some original vertices of $G$. Let $P'$ be the path in $G$ obtained from $P$ as follows. If $P$ includes the entire subdivision of some edge $vw$ of $G$ then replace that subpath by $vw$ in $P'$. If $P$ includes a subpath of the subdivision of some edge $vw$ of $G$, then without loss of generality, it includes $v$, in which case replace that subpath by $v$ in $P'$. Since the colours assigned to division vertices are distinct from the colours assigned to original vertices, a $t$-vertex path of division vertices in the first half of $P$ corresponds to a $t$-vertex path of division vertices in the second half of $P$. Hence $P'$ is a repetitively coloured path in $G$. This contradiction proves that $H$ is non-repetitively coloured. Hence $\pi(H)\leq k+3$.
\end{proof}

Note that \lemref{NonRepSub}(a) is best possible in the weak sense that $\pi(C_5)=4$ and $\pi(C_4)=3$; see \citep{Currie-TCS05}.

Loosely speaking, \lemref{NonRepSub} says that non-repetitive colourings of subdivisions are not much ``harder" than non-repetitive colourings of the original graph. This intuition is made more precise if we subdivide each edge many times. Then non-repetitive colourings of subdivisions are much ``easier" than non-repetitive colourings of the original graph. In particular, \citet{Gryczuk-IJMMS07} proved that every graph has a non-repetitively 5-colourable subdivision. This bound was improved to 4 by \citet{BaratWood-EJC08} and by \citet{MarxSchaefer-DAM09}, and very recently to 3 by \citet{PezZma}; see \citep{BGKNP-NonRepTree-DM07,Currie-TCS05} for related results. This implies that the class of non-repetitively $3$-colourable graphs is not contained in a proper topologically-closed class.


We now set out to prove a converse of \lemref{NonRepSub}; that is, $\pi(G)$ is bounded by a function of $\pi(H)$. The following tool by \citet{NR-JCTB00} will be useful.


\begin{lemma}[\citep{NR-JCTB00}]\lemlabel{NesRas}
For every $k$-colouring of the arcs of an oriented forest $T$, there is a $(2k+1)$-colouring of the vertices of $T$, such that between each pair of (vertex) colour classes, all arcs go in the same direction and have the same colour. 
\end{lemma}

A \emph{rooting} of a forest $F$ is obtained by nominating one vertex in each component tree of $F$ to be a \emph{root} vertex.

\begin{lemma}\lemlabel{NonRepForest}
Let $T'$ be the $1$-subdivision of a forest $T$, such that $\pi(T')\leq k$. Then $$\pi(T)\leq k(k+1)(2k+1).$$ 
Moreover, for every non-repetitive $k$-colouring $c$ of $T'$, and for every rooting of $T$, there is a non-repetitive $k(k+1)(2k+1)$-colouring $q$ of $T$, such that:
\begin{enumerate}
\item[(a)] For all edges $vw$ and $xy$ of $T$ with $q(v)=q(x)$ and $q(w)=q(y)$, the division vertices corresponding to $vw$ and $xy$ have the same colour in $c$.
\item[(b)]  For all non-root vertices $v$ and $x$ with $q(v)=q(x)$, the division vertices corresponding to the parent edges of $v$ and $x$ have the same colour in $c$.
\item [(c)] For every root vertex $r$ and every non-root vertex $v$, we have $q(r)\neq q(v)$.
\item [(d)] For all vertices $v$ and $w$ of $T$, if $q(v)=q(w)$ then $c(v)=c(w)$.
\end{enumerate}
\end{lemma}

\begin{proof}
Let $c$ be a non-repetitive $k$-colouring of $T'$, with colours $[1,k]$.
Colour each edge of $T$ by the colour assigned by $c$ to the corresponding division vertex. 
Orient each edge of $T$ towards the root vertex in its component.
By \lemref{NesRas}, there is a $(2k+1)$-colouring $f$ of the vertices of $T$, such that between each pair of (vertex) colour classes in $f$, all arcs go in the same direction and have the same colour in $c$. 
Consider a vertex $v$ of $T$. 
If $v$ is a root, let $g(r):=0$; otherwise let $g(v):=c(vw)$ where $w$ is the parent of $v$.
Let $q(v):=(c(v),f(v),g(v))$. 
The number of colours in $q$ is at most $k(k+1)(2k+1)$.
Observe that claims (c) and (d)  hold by definition.

We claim that $q$ is non-repetitive.
Suppose on the contrary that there is a path $P=(v_1,\dots,v_{2s})$ in $T$ that is repetitively coloured by $q$. 
That is, $q(v_i)=q(v_{i+s})$  for each $i\in[1,k]$. Thus $c(v_i)=c(v_{i+s})$ and $f(v_i)=f(v_{i+s})$ and $g(v_i)=g(v_{i+s})$. Since no two root vertices are in a common path, (c) implies that every vertex in $P$ is a non-root vertex.

Consider the edge $v_iv_{i+1}$ of $P$ for some $i\in[1,s-1]$.
We have  $f(v_i)=f(v_{i+s})$ and  $f(v_{i+1})=f(v_{i+s+1})$. 
Between these two colour classes in $f$, all arcs go in the same direction and have the same colour. 
Thus the edge $v_iv_{i+1}$ is oriented from $v_i$ to $v_{i+1}$ if and only if  the edge $v_{i+s}v_{i+s+1}$ is oriented from $v_{i+s}$ to $v_{i+s+1}$. 
And $c(v_iv_{i+1})=c(v_{i+s}v_{i+s+1})$.

If at least two vertices $v_i$ and $v_j$ in $P$ have indegree $2$ in $P$, then some vertex between $v_i$ and $v_j$ in $P$ has outdegree $2$ in $P$, which is a contradiction. Thus at most one vertex has indegree $2$ in $P$.  
Suppose that $v_i$ has indegree $2$ in $P$. Then each edge $v_jv_{j+1}$ in $P$ is oriented from $v_j$ to $v_{j+1}$ if $j\leq i-1$, and from $v_{j+1}$ to $v_{j}$ if $j\geq i$ (otherwise two vertices have indegree $2$ in $P$). 
In particular, $v_1v_2$ is oriented from $v_1$ to $v_2$ and $v_{s+1}v_{s+2}$ is oriented from $v_{s+2}$ to $v_{s+1}$. 
This is a contradiction since the edge $v_1v_2$ is oriented from $v_1$ to $v_2$ if and only if  the edge $v_{s+1}v_{s+2}$ is oriented from $v_{s+1}$ to $v_{s+2}$.
Hence no vertex in $P$ has indegree $2$. Thus $P$ is a directed path.

Without loss of generality, $P$ is oriented from $v_1$ to $v_{2s}$. 
Let $x$ be the parent of $v_{2s}$. 
Now  $g(v_{2s})=c(v_sx)$ and $g(v_s)=c(v_sv_{s+1})$ and $g(v_s)=g(v_{2s})$.
Thus $c(v_{s}v_{s+1})=c(v_{2s}x)$. 

Summarising, the path 
\begin{align*}
\big(\underbrace{ v_1,v_1v_2,v_2,\dots,v_s, v_sv_{s+1}},
\underbrace{ v_{s+1},v_{s+1}v_{s+2}, v_{s+2}, \dots, v_{2s}, v_{2s}x}\big)
\end{align*}
in $T'$ is repetitively coloured by $c$. (Here division vertices in $T'$ are described by the corresponding edge.)\ Since $c$ is non-repetitive in $T'$, we have the desired contradiction. Hence $q$ is a non-repetitive colouring of $T$.

It remains to prove claims (a) and (b). Consider two edges $vw$ and $xy$ of $T$, such that $q(v)=q(x)$ and $q(w)=q(y)$. Thus $f(v)=f(x)$ and $f(w)=f(y)$. Thus $vw$ and $xy$ have the same colour in $c$. Thus the division vertices corresponding to $vw$ and $xy$ have the same colour in $c$. This proves claim (a). Finally consider  non-root vertices $v$ and $x$ with $q(v)=q(x)$. Thus $g(v)=g(x)$. Say $w$ and $y$ are the respective parents of $v$ and $x$. By construction, $c(vw)=c(xy)$. Thus the division vertices of $vw$ and $xy$  have the same colour in $c$. This proves claim (b).
\end{proof}


We now extend \lemref{NonRepForest} to apply to graphs with bounded acyclic chromatic number; see \citep{AM-JAC98,NR-JCTB00} for similar methods. 

\begin{lemma}\lemlabel{NonRepGraph}
Let $G'$ be the $1$-subdivision of a graph $G$, such that $\pi(G')\leq k$ and $\acy(G)\leq\ell$. Then 
$$\pi(G)\leq \ell\big(k(k+1)(2k+1)\big)^{\ell-1}.$$
\end{lemma}

\begin{proof} 
Let $p$ be an acyclic $\ell$-colouring of $G$, with colours $[1,\ell]$.
Let $c$ be a non-repetitive $k$-colouring of $G'$.
For distinct $i,j\in[1,\ell]$, let $G_{i,j}$ be the subgraph of $G$ induced by the vertices coloured $i$ or $j$ by $p$. 
Thus each $G_{i,j}$ is a forest, and $c$ restricted to $G'_{i,j}$ is non-repetitive.

Apply \lemref{NonRepForest} to each $G_{i,j}$. Thus $\pi(G_{i,j})\leq k(k+1)(2k+1)$, and 
there is a non-repetitive $ k(k+1)(2k+1)$-colouring $q_{i,j}$ of $G_{i,j}$ satisfying  \lemref{NonRepForest}(a)--(d).

Consider a vertex $v$ of $G$. For each colour $j\in[1,\ell]$ with $j\neq p(v)$, let $q_j(v):=q_{p(v),j}(v)$.
Define $$q(v):=\Big(p(v),\big\{(j,q_{j}(v)):j\in[1,\ell],j\neq p(v)\big\}\Big).$$ 
Note that the number of colours in $q$ is at most $\ell\big( k(k+1)(2k+1)\big)^{\ell-1}.$ 
We claim that $q$ is a non-repetitive colouring of $G$. 

Suppose on the contrary that some path $P=(v_1,\dots,v_{2s})$ in $G$ is repetitively coloured by $q$. That is, $q(v_a)=q(v_{a+s})$ for each $a\in[1,s]$. 
Thus $p(v_a)=p(v_{a+s})$ and  for each $a\in[1,s]$.
Let $i:=p(v_a)$. Choose any $j\in[1,\ell]$ with $j\neq i$. 
Thus $(j,q_j(v_a))=(j,q_j(v_{a+s}))$ and $q_j(v_a)=q_j(v_{a+s})$. 
Hence $c(v_a)=c(v_{a+s})$  by \lemref{NonRepForest}(d).

Consider an edge $v_av_{a+1}$ for some $i\in[1,s-1]$. 
Let $i:=p(v_a)$ and $j:=p(v_{a+1})$. 
Now $q(v_a)=q(v_{a+s})$ and $q(v_{a+1})=q(v_{a+s+1})$. 
Thus $p(v_{a+s})=i$ and $p(v_{a+s+1})=j$. 
Moreover, $(j,q_j(v_a))=(j,q_j(v_{a+s}))$ and $(i,q_i(v_{a+1}))=(i,q_i(v_{a+s+1}))$. 
That is, $q_{i,j}(v_a)=q_{i,j}(v_{a+s})$ and $q_{i,j}(v_{a+1})=q_{i,j}(v_{a+s+1})$. 
Thus $c(v_av_{a+1})=c(v_{a+s}v_{a+s+1})$ by \lemref{NonRepForest}(a).

Consider the edge $v_sv_{s+1}$. Let $i:=p(v_s)$ and $j:=p(v_{s+1})$. 
Without loss of generality, $v_{s+1}$ is the parent of $v_s$ in the forest $G_{i,j}$. 
In particular, $v_s$ is not a root of $G_{i,j}$. 
Since $q_{i,j}(v_s)=q_{i,j}(v_{2s})$ and by \lemref{NonRepForest}(c),
$v_{2s}$ also is not a root of $G_{i,j}$.
Let $y$ be the parent of $v_{2s}$ in $G_{i,j}$. 
By \lemref{NonRepForest}(b) applied to $v_s$ and $v_{2s}$,  we have $c(v_sv_{s+1})=c(v_{2s}y)$.

Summarising, the path
\begin{align*}
\big(\underbrace{ v_1,v_1v_2,v_2,\dots,v_s, v_sv_{s+1}},
\underbrace{ v_{s+1},v_{s+1}v_{s+2}, v_{s+2}, \dots, v_{2s}, v_{2s}y}\big)
\end{align*}
is repetitively coloured in $G'$. This contradiction proves that $G$ is repetitively coloured by $q$.
\end{proof} 

\twolemref{NonRepGraph}{NonRepSub}(a) imply:

\begin{lemma}\lemlabel{NonRepGraphs}
Let $H$ be a $(\leq 1)$-subdivision of a graph $G$, such that $\pi(H)\leq k$ and $\acy(G)\leq\ell$. 
Then $$\pi(G)\leq \ell\big((k+1)(k+2)(2k+3)\big)^{\ell-1}.$$
\end{lemma}

\begin{lemma}\lemlabel{ConstructAcyclic}
Let $c$ be a non-repetitive $k$-colouring of the 1-subdivision $G'$ of a graph $G$.
Then 
$$\acy(G) \leq k\cdot 2^{2k^2}.$$
\end{lemma}

\begin{proof} 
Orient the edges of $G$ arbitrarily. Let $A(G)$ be the set of oriented arcs of $G$.
So $c$ induces a $k$-colouring of $V(G)$ and of $A(G)$.
For each vertex $v$ of $G$, let
  $$q(v):=\big\{c(v)\big\} \cup \big\{ (+, c(vw), c(w) ) : vw \in A(G) \big\} \cup 
  		\big\{ (-, c(wv), c(w) ) : wv \in A(G) \big\}.$$
The number of possible values for $q(v)$ is at most $k\cdot 2^{2k^2}$.
We claim that $q$ is an acyclic colouring of $G$.

Suppose on the contrary that $q(v) = q(w)$ for some arc $vw$ of $G$. 
Thus $c(v)=c(w)$ and $(+,c(vw), c(w) ) \in q(v)$, implying $(+,c(vw) , c(w) ) \in q(w)$. 
That is, for some arc $wx$, we have $c(wx)=c(vw)$ and $c(x)=c(w)$. 
Thus the path $(v,vw,w,wx)$ in $G'$ is repetitively coloured. 
This contradiction shows that $q$ properly colours $G$.

It remains to prove that $G$ contains no bichromatic cycle (with respect to  $q$). First consider a bichromatic path $P=(u,v,w)$ in $G$ with $q(u)=q(w)$. Thus $c(u)=c(w)$.

Suppose on the contrary that $P$ is oriented $(u,v,w)$, as illustrated in \figref{ConstructAcyclic}(a). 
By construction, $(+,c(uv),c(v)) \in q(u)$, implying  $(+,c(uv),c(v)) \in q(w)$. 
That is, $c(uv)=c(wx)$ and $c(v)=c(x)$ for some arc $wx$ (and thus $x \neq v$). 
Similarly, $(-,c(vw),c(v)) \in q(w)$, implying  $(-,c(vw),c(v)) \in q(u)$. 
Thus $c(vw)=c(tu)$ and $c(v)=c(t)$ for some arc $tu$ (and thus $t \neq v$). 
Hence the 8-vertex path $(tu,u,uv,v,vw,w,wx,x)$ in $G'$ is repetitively coloured by $c$, as illustrated in \figref{ConstructAcyclic}(b). 
This contradiction shows that both edges in $P$ are oriented toward $v$ or both are oriented away from $v$.

Consider the case in which both edges in $P$ are oriented toward $v$.
Suppose on the contrary that $c(uv)\neq c(wv)$.
By construction, $(+,c(uv),c(v))\in q(u)$, implying $(+,c(uv),c(v))\in q(w)$.
That is, $c(uv)=c(wx)$ and $c(v)=c(x)$ for some arc $wx$ (implying $x \neq v$ since $c(uv)\neq c(wv)$). 
Similarly, $(+,c(wv),c(v))\in q(w)$, implying $(+,c(wv),c(v))\in q(u)$.
That is, $c(wv)=c(ut)$ and $c(t)=c(v)$ for some arc $ut$ (implying $t \neq v$ since $c(ut)=c(wv)\neq c(uv)$). 
Hence the path $(ut,u,uv,v,wv,w,wx,x)$ in $G' $ is repetitively coloured in $c$, as illustrated in \figref{ConstructAcyclic}(c). 
This contradiction shows that $c(uv)=c(wv)$.
By symmetry, $c(uv)=c(wv)$ when both edges in $P$ are oriented away from $v$.

Hence in each component of $G'$, all the division vertices have the same colour in $c$.
Every bichromatic cycle contains a 4-cycle or a 5-path. 
If  $G$ contains a bichromatic 5-path $(u,v,w,x,y)$, then all the division vertices in $(u,v,w,x,y)$ have the same colour in $c$, and  $(u,uv,v,vw,w,wx,x,xy)$ is a repetitively coloured path in $G'$, as illustrated in \figref{ConstructAcyclic}(d). 
Similarly, if $G$ contains a bichromatic 4-cycle $(u,v,w,x)$, then  all the division vertices in $(u,v,w,x)$  have the same colour in $c$, and $(u,uv,v,vw,w,wx,x,xu)$ is a repetitively coloured path in $G'$, as illustrated in \figref{ConstructAcyclic}(e). 
 
Thus $G$ contains no bichromatic cycle, and $q$ is an acyclic colouring of $G$.
\end{proof}

\Figure{ConstructAcyclic}{\includegraphics[width=\textwidth]{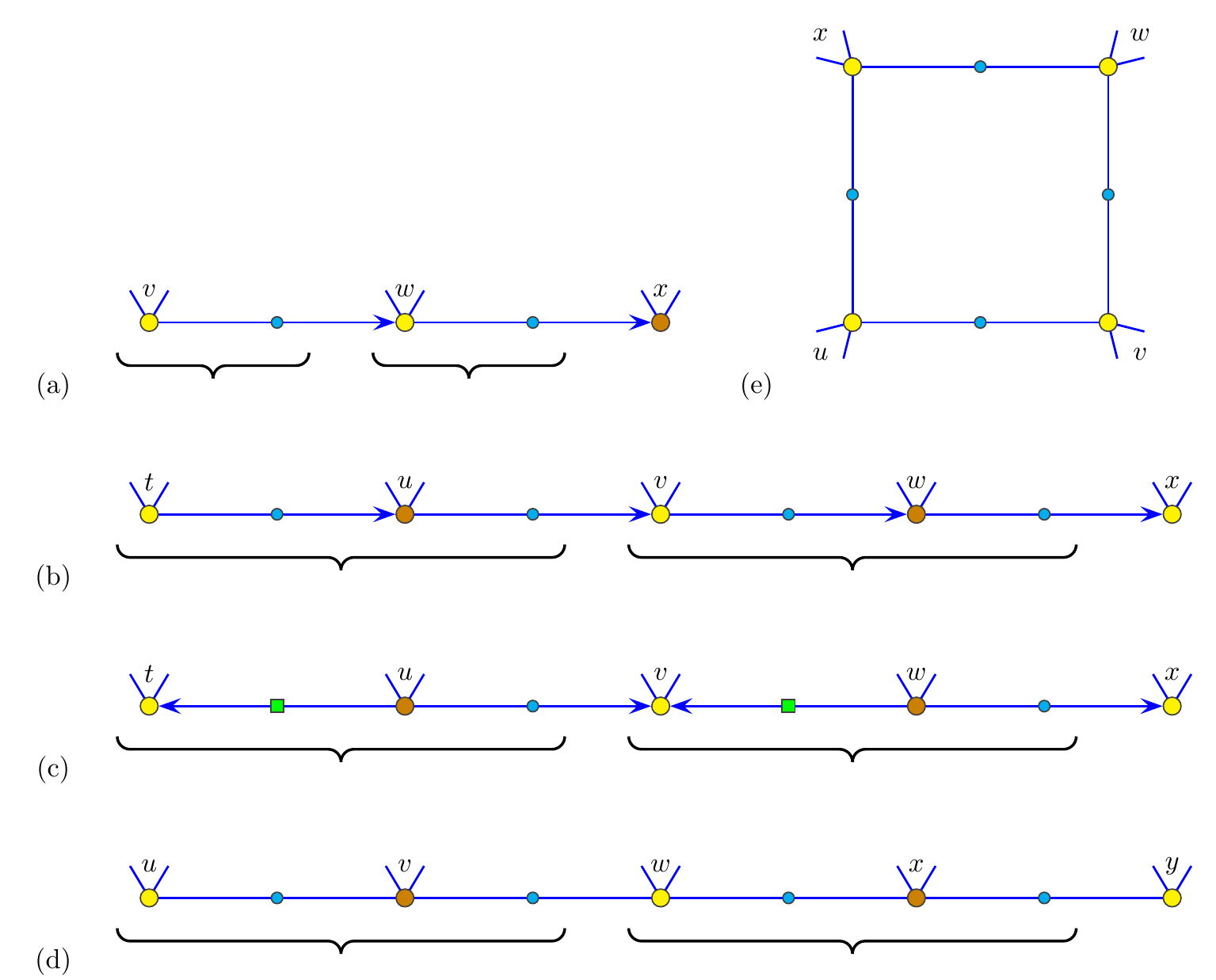}}{Illustration for \lemref{ConstructAcyclic}.}

Note that the above proof establishes the following stronger statement: 
If the 1-subdivision of a graph $G$ has a $k$-colouring that is non-repetitive on paths with at most $8$ vertices, then $G$ has an acyclic $k\cdot 2^{2k^2}$-colouring in which each component of each 2-coloured subgraph is a star or a 4-path.

\twolemref{ConstructAcyclic}{NonRepSub}(a) imply:

\begin{lemma}\lemlabel{ConstructAcyclicc}
If some $(\leq1)$-subdivision of a graph $G$ has a non-repetitive $k$-colouring, then 
$\acy(G) \leq(k+1)\cdot 2^{2(k+1)^2}$.
\end{lemma}

\begin{lemma}\lemlabel{Subdiv2Graph}
If $\pi(H)\leq k$ for some $(\leq 1)$-subdivision of a graph $G$, then 
$$\pi(G)\leq (k+1)\cdot 2^{2(k+1)^2}\big((k+1)(k+2)(2k+3)\big)^{(k+1)\cdot 2^{2(k+1)^2}-1}.$$
\end{lemma}

\begin{proof}
$\acy(G)\leq(k+1)\cdot2^{2(k+1)^2}$ by \lemref{ConstructAcyclicc}. 
The result follows from \lemref{NonRepGraphs} with $\ell=(k+1)\cdot 2^{2(k+1)^2}$.
\end{proof}

\begin{corollary}\corlabel{GenSubdiv2Graph}
There is a function $f$ such that $\pi(G)\leq f(\pi(H),d)$ for every $(\leq d)$-subdivision $H$ of a graph $G$.
\end{corollary}

One of the most interesting open problems regarding non-repetitive colourings is whether planar graphs have bounded $\pi$ (as mentioned in most papers regarding non-repetitive colourings). \coref{GenSubdiv2Graph} implies that to prove that  planar graphs have bounded $\pi$ it suffices to show that every planar graph has a subdivision with bounded $\pi$ and a bounded number of division vertices per edge. This shows that Conjectures 4.1 and 5.2 in \citep{Gryczuk-IJMMS07} are equivalent.

We now get to the main results of this section. \twolemref{Subdiv2Graph}{NonRepSub}(a) imply:

\begin{theorem}\thmlabel{NonRepTopo}
$\pi$ is strongly topological.
\end{theorem}

$\pi$ is degree-bound since every graph $G$ has a vertex of degree at most $2\pi(G)-2$; see \citep[Proposition 5.1]{BaratWood-EJC08}. Since $\pi$ is hereditary, \lemref{TopoExp} and \thmref{NonRepTopo} imply:

\begin{theorem}
For every constant $c$ the class of graphs $\{G:\pi(G)\leq c\}$ has bounded expansion.
\end{theorem}

\subsection{Subdivisions of Complete Graphs}

\coref{GenSubdiv2Graph} with $G=K_n$ implies that there is a function $f$ such that for every $(\leq d)$-subdivision $H$ of $K_n$, 
$$\pi(H)\geq f(n,d),$$ 
and  $\lim_{n\rightarrow\infty}f(n,d)=\infty$ for all fixed $d$. We now obtain reasonable bounds on $f$. 

\begin{lemma}\lemlabel{Knd}
Let $K_{n,d}$ be the $d$-subdivision of $K_n$. Then 
$$\pi(K_{n,d})\geq \left(\frac{n}{2}\right)^{1/(d+1)}.$$
\end{lemma}

\begin{proof}
Suppose on the contrary that $c=\pi(K_{n,d})<\left(\frac{n}{2}\right)^{1/(d+1)}$. 
Fix a non-repetitive $c$-colouring of $K_{n,d}$. 
Orient each edge of $K_n$ arbitrarily. Colour each arc $vw$ of $K_n$ by the $d$-vector of colours assigned to the division vertices on the path from $v$ to $w$ in $K_{n,d}$.
The number of arc colours is at most $c^d$.
Let $p:=\ceil{\frac{n}{c}}$. 
There is a $K_p$ subgraph of $K_n$ whose vertices are monochromatic in $q$, 
and there is a subgraph $H$ of $K_p$ consisting of at least $\binom{p}{2}/c^d$ monochromatic arcs. 
Now $p\geq\frac{n}{c}>2c^d$. Thus $p-1\geq2c^d$ and $\binom{p}{2}/c^d\geq p$.
Hence  $H$ has at least $p$ arcs. 

If $H$ contains a vertex $v$ with an incoming arc $uv$ and an outgoing arc $vw$, then $K_{n,d}$ contains a repetitively coloured path on $2d+2$ vertices, as illustrated in \figref{PathCyclePath}(a). 
Thus for every vertex $v$ of $H$, all the arcs incident to $v$ are incoming or all are outgoing. 
In particular, $H$ has no triangle. 
Since $H$ has at least $p$ arcs, the  undirected graph underlying $H$ contains a cycle. 
If $H$ contains a 4-cycle then $K_{n,d}$ contains a repetitively coloured path on $4d+4$ vertices, as illustrated in \figref{PathCyclePath}(b). 
Otherwise the undirected graph underlying  $H$ contains a $5$-vertex path, in which case,  $K_{n,d}$ contains a 
repetitively coloured path on $4d+4$ vertices, as illustrated in \figref{PathCyclePath}(c). This is the desired contradiction.
\end{proof}

\Figure{PathCyclePath}{\includegraphics[width=\textwidth]{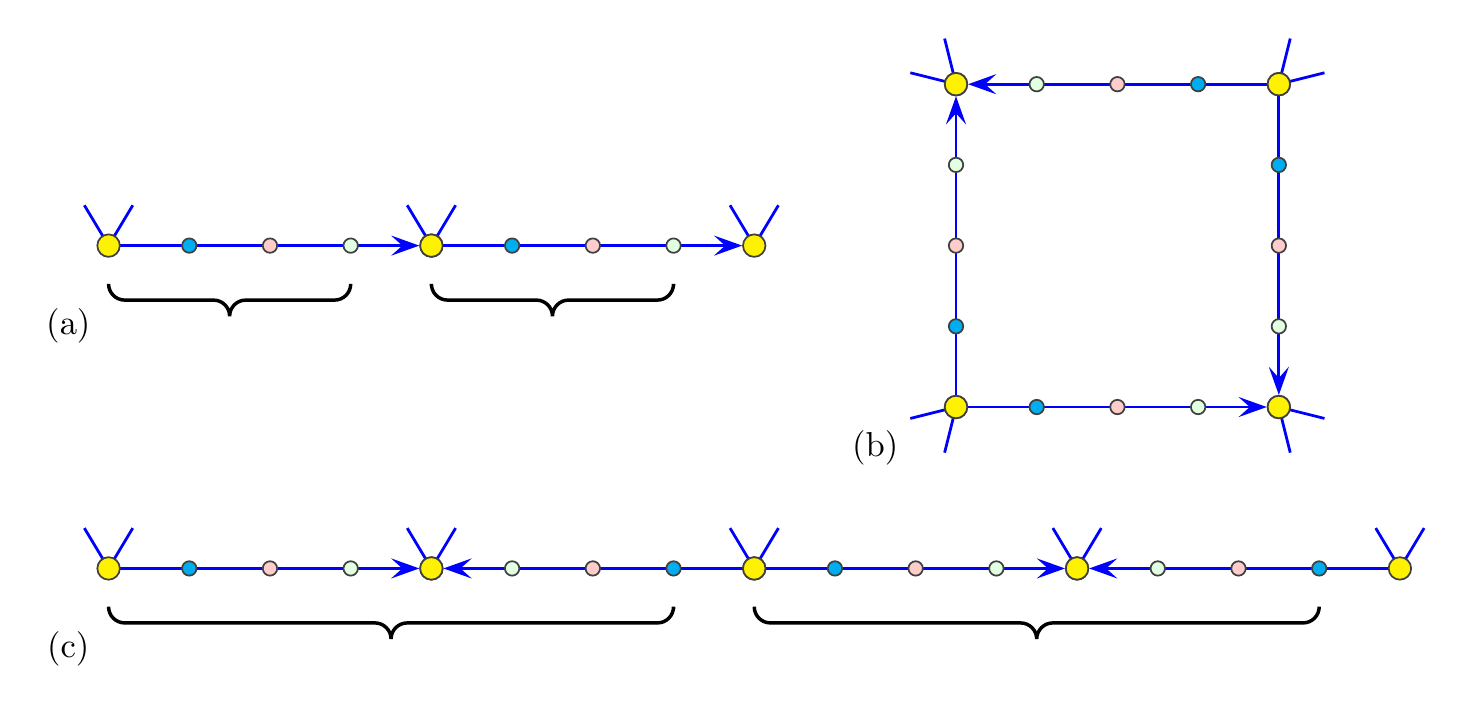}}{Illustration for \lemref{Knd}.}

\twolemref{Knd}{NonRepSub}(c) imply:

\begin{corollary}
\corlabel{GeneralCompleteSubdivision}
If $H$ is a $(\leq d)$-subdivision of $K_n$ then  
$$\pi(H)\geq  \left(\frac{n}{2}\right)^{1/(d+1)}-3.$$
\end{corollary}

Determining $\pi(K_n')$ is an interesting open problem. The lower bound  $\pi(K_n')\geq\sqrt{n}$ follows from a result by \citet{BreakingRhythm}, and also follows from the previously mentioned lower bound $\pi(K_n')\geq\st(K_n')\geq\sqrt{n}$ by  \citet{Wood-DMTCS05}. Here is the best known upper bound.

\begin{proposition}\proplabel{CompleteGraphFirstSubdiv}
$\pi(K_n')\leq \frac{3}{2}n^{2/3}+O(n^{1/3})$.
\end{proposition}

\begin{proof}
Let $N:=\ceil{n^{1/3}}$. In $K_{N^3}'$, let 
$\{\langle i,k\rangle: 1\leq i\leq N^2,1\leq k\leq N\}$ be the original vertices, 
and let $\langle i,k;j,\ell\rangle$ be the division vertex having
$\langle i,k\rangle$  and $\langle j,\ell\rangle$  as its neighbours.

Colour each original vertex \blah{i,j} by $A_i$.
Colour each division vertex \blah{i,k;j,\ell} by $B_k$ if $i<j$.
Colour each division vertex \blah{i,k;i,\ell} by $C_{k,\ell}$ where $k<\ell$.

Suppose that $PQ$ is a repetitively coloured path. By parity, $|P|$ is even.

First suppose that $|P|\geq 4$. Then $P$ contains some transition $T$. Observe that each transition is uniquely identified by the three colours that it receives. In particular, the only transition coloured $A_iB_kA_j$ with $i<j$ is $\langle i,k\rangle\langle i,k;j,\ell\rangle\langle j,\ell\rangle$. And the only transition coloured $A_iC_{k,\ell}A_i$ is $\langle i,k\rangle\langle i,k;i,\ell\rangle\langle i,\ell\rangle$. Thus $T$ is repeated in $Q$, which is a contradiction.

Otherwise $|P|=2$. Thus $PQ$ is coloured $A_iC_{k,\ell}A_iC_{k,\ell}$ for some $k<\ell$. But the only edges coloured $A_iC_{k,\ell}$ are the two edges in the transition $\langle i,k\rangle\langle i,k;i,\ell\rangle\langle i,\ell\rangle$, which again is a contradiction.

Hence there is no repetitively coloured path. The number of colours is $N^2+N+\binom{N}{2}\leq\frac{3}{2}N^2+O(N)\leq\frac{3}{2}n^{2/3}+O(n^{1/3})$.
\end{proof}


We now determine $\pi(K_{n,d})$ to within a constant factor.

\begin{lemma}
\lemlabel{CompleteGraphSubdiv}
Let $A\geq1$ and $B\geq2$ and $d\geq2$ be integers. If $n\leq A\cdot B^d$ then $\pi(K_{n,d})\leq A+8B$.
\end{lemma}

\begin{proof}
Let $(c_1,\dots,c_d)$ be a non-repetitive sequence such that $c_1=0$ and $\{c_2,c_3,\dots,c_d\}\subseteq\{1,2,3\}$. Let $\preceq$ be a total ordering of the original vertices of $K_{n,d}$. Since $n\leq A\cdot B^d$, the original vertices of $K_{n,d}$ can be labelled $$\{v=\blah{v_0,v_1,\dots,v_d}:1\leq v_0\leq A,1\leq v_i\leq B,1\leq i\leq d\}.$$ 

Colour each original vertex $v$ by $\col(v):=v_0$. 
Consider a pair of original vertices $v$ and $w$ with $v\prec w$.
If $(v,r_1,r_2,\dots,r_d,w)$ is the transition from $v$ to $w$, then for $i\in[1,d]$, colour the division vertex $r_i$ by $$\col(r_i):=(\delta(v_i,w_i),c_i,v_i),$$ where
$\delta(a,b)$ is the indicator function of $a=b$. 
We say this transition is \emph{rooted} at $v$.
Observe that the number of colours is at most $A+2\cdot4\cdot B=A+8B$.

Every transition is coloured $$\big(x_0,(\delta_1,c_1,x_1),(\delta_2,c_2,x_2),\dots,(\delta_d,c_d,x_d),x_{d+1}\big)$$
for some $x_0\in[1,A]$ and $x_1,\dots,x_{d+1}\in[1,B]$ and $\delta_1,\dots,\delta_d\in\{\text{true},\text{false}\}$.
Every such transition is rooted at the original vertex $\blah{x_0,x_1,\dots,x_d}$. That is, the colours assigned to a transition determine its root.

Suppose on the contrary that $P=(a_1,\dots,a_{2s})$ is a repetitively coloured path in $K_{n,d}$. Since every original vertex receives a distinct colour from every division vertex, for all $i\in[s]$, $a_i$ is an original vertex if and only if $a_{i+s}$ is an original vertex, and $a_i$ is a division vertex if and only if $a_{i+s}$ is a division vertex. 

By construction, every transition is coloured non-repetitively. 
Thus $P$ contains at least one original vertex, implying $\{a_1,\dots,a_s\}$ contains at least one original vertex. 
If $\{a_1,\dots,a_s\}$ contains at least two original vertices, then $\{a_1,\dots,a_s\}$ contains a transition $(a_i,\dots,a_{i+d+1})$, implying $(a_{s+i},\dots,a_{s+i+d+1})$ is another transition receiving the same tuple of colours. Thus $(a_i,\dots,a_{i+d+1})$ and $(a_{s+i},\dots,a_{s+i+d+1})$ are rooted at the same original vertex, implying $P$ is not a path. 

Now assume there is exactly one original vertex $a_i$ in $\{a_1,\dots,a_s\}$. Thus $a_{s+i}$ is the only original vertex in $\{a_{s+1},\dots,a_{2s}\}$. Hence $(a_i,\dots,a_{s+i})$ is a transition, implying $s=d+1$. Without loss of generality, $a_i\prec a_{s+i}$ and this transition is rooted at $a_i$. 

Let $v:=a_i$ and $w:=a_{s+i}$. For $j\in[1,d]$, the vertex $a_{i+j}$ is the $j$-th vertex in the transition from $v$ to $w$, and is thus coloured $(\delta(v_j,w_j),c_j,v_j)$.

Suppose that $i\leq s-1$. Let $x$ be the original vertex such that the transition between $w$ and $x$ contains $\{a_{s+i+1},\dots,a_{2s}\}$. 
Now $$\col(a_{s+i+1})=\col(a_{i+1})=(\delta(v_1,w_1),c_1,v_1).$$
Since $c_1\neq c_d$, we have $w\prec x$.
For $j\in[1,s-i]$, the vertex $a_{s+i+j}$ is the $j$-th vertex in the transition from $w$ to $x$, and thus $$(\delta(w_j,x_j),c_j,w_j)=\col(a_{s+i+j})=
\col(a_{i+j})=(\delta(v_j,w_j),c_j,v_j).$$ 
In particular, $v_j=w_j$ for all $j\in[1,s-i]$. Note that if $i=s$ then this conclusion is vacuously true.

Now suppose that $i\geq 2$. Let $u$ be the original vertex such that the transition between $u$ and $v$ contains $\{a_1,\dots,a_{i-1}\}$. 
Now $$\col(a_{i-1})=\col(a_{s+i-1})=(\delta(v_d,w_d),c_d,v_d).$$
Since $c_d\neq c_1$, we have $u\prec v$. 
For $j\in[s-i+1,d]$, the vertex $a_{i+j-s}$ is the $j$-th vertex in the transition from $u$ to $v$, and thus $$(\delta(u_j,v_j),c_j,u_j)=\col(a_{i+j-s})=\col(a_{i+j})=(\delta(v_j,w_j),c_j,v_j).$$ In particular, $v_j=u_j$ and $\delta(v_j,w_j)=\delta(u_j,v_j)$. Thus $v_j=w_j$ for all $j\in[s-i+1,d]$. Note that if $i=1$ then this conclusion is vacuously true.

Hence $v_j=w_j$ for all $j\in[1,d]$. 
Now $v$ is coloured $v_0$, and $w$ is coloured $w_0$. 
Since $v=a_i$ and $w=a_{s+i}$ receive the same colour, $v_0=w_0$.
Therefore $v_j=w_j$ for all $j\in[0,d]$. That is, $v=w$, which is the desired contradiction. 

Therefore there is no repetitively coloured path in $K_{n,d}$.
\end{proof}


\begin{theorem}
\thmlabel{CompleteGraphSubdiv}
For $d\geq2$, 
$$\left(\frac{n}{2}\right)^{1/(d+1)}\leq \pi(K_{n,d})\leq 9\ceil{n^{1/(d+1)}}.$$
\end{theorem}

\begin{proof}
The lower bound is \lemref{Knd}. The upper bound is \lemref{CompleteGraphSubdiv} with $B=(n/8)^{1/(d+1)}$ and $A=8B$.
\end{proof}

As mentioned earlier, $K_n$ has a subdivision $H$ with $\pi(H)\leq\Oh{1}$. All known constructions of $H$ use at least $\Omega(n)$ division vertices on some edge---some use $\Omega(n^2)$ division vertices on every edge. We now show that $\Theta(\log n)$ division vertices is best possible. 

\begin{theorem}
The $\ceil{\log n}$-subdivision of $K_n$ has a non-repetitive $17$-colouring. Moreover, if $H$ is a subdivision of $K_n$ and $\pi(H)\leq c$ then some edge of $K_n$ is subdivided at least $\log_{c+3}(\frac{n}{2})-1$ times.
\end{theorem}

\begin{proof}
The upper bound follows from \lemref{CompleteGraphSubdiv} with $A=1$ and $B=2$.  (Note that the bound of $17$ can be easily lowered with a little more proof.)\ For the lower bound, suppose that  $H$ is a $(\leq d)$-subdivision of $K_n$ and $\pi(H)\leq c$. By \coref{GeneralCompleteSubdivision}, $(\frac{n}{2})^{1/(d+1)}-3\leq \pi(H)\leq c$.
That is, $\log_{c+3}\frac{n}{2}-1\leq d$.
Hence some edge of $H$ is subdivided at least 
$\log_{c+3}(\frac{n}{2})-1$ times.
\end{proof}

We conclude this section with a final observation about repetitive colourings. \citet{DujWoo-Order06} defined a \emph{strong star colouring} of a graph $G$ to be a colouring such that  between each pair of colour classes all the edges are incident to a single vertex. That is, each bichromatic subgraph is a star plus isolated vertices. If $(v_1,\dots,v_{2s})$ is a repetitively coloured path, then $v_1v_2$ and $v_{s+1}v_{s+2}$ are bichromatic edges with no vertex in common. Thus every strong star colouring is non-repetitive.

\section*{Acknowledgements}
The authors would like to thank Vida Dujmovi\'c for simplifying a clumsy proof in an early draft of the paper.


\def\Dbar{\leavevmode\lower.6ex\hbox to 0pt{\hskip-.23ex \accent"16\hss}D}

\end{document}